\documentclass{amsart}
\usepackage{amsmath,amssymb,amsthm,array,young,environ}
\usepackage[matrix,arrow,curve]{xy}

\newtheorem{cor}{Corollary}
\newtheorem{lem}{Lemma}
\newtheorem{prop}{Proposition}
\newtheorem{propconstr}{Proposition-Construction}

\newtheorem{conj}{Conjecture}
\newtheorem{thm}{Theorem}

\theoremstyle{remark}
\newtheorem{Rem}{Remark}

\newcommand{\FD}{\mathsf{FD}}
\newcommand {\Fl}{\mathrm{Fl}}
\newcommand {\GL}{\mathrm{GL}}

\DeclareMathOperator {\im}{im}
\DeclareMathOperator {\id}{id}
\DeclareMathOperator {\Ad}{Ad}
\DeclareMathOperator {\End}{End}
\DeclareMathOperator {\Hom}{Hom}
\DeclareMathOperator {\RSK}{RSK_{mir}}
\DeclareMathOperator {\mir}{mir}
\DeclareMathOperator {\Sh}{Sh}

\newcommand {\eql}[2]{\begin{equation}\label{#1}#2\end{equation}}

\newcommand {\F}{\mathrm F}

\newcommand {\AAbb}{\mathbb A}
\newcommand {\CC}{\mathbb C}
\newcommand {\FF}{\mathbb F}
\newcommand {\GG}{\mathbb G}
\newcommand {\PP}{\mathbb P}
\newcommand {\QQ}{\mathbb Q}
\newcommand{\BZ}{{\mathbb Z}}

\newcommand {\gl}{\mathfrak {gl}}
\newcommand {\gS}{\mathfrak S}
\newcommand{\fP}{{\mathfrak P}}

\newcommand {\bH}{\mathbf H}

\newcommand {\bR}{\mathbf R}
\newcommand {\bT}{\mathbf T}
\newcommand{\bn}{{\mathbf n}}
\newcommand{\bq}{{\mathbf q}}
\newcommand{\bt}{{\mathbf t}}
\newcommand{\bv}{{\mathbf v}}

\renewcommand {\H}{\mathcal H}
\newcommand {\N}{\mathcal N}
\newcommand {\M}{\mathcal M}
\newcommand {\R}{\mathcal R}
\newcommand{\Olm}{{\mathcal O}_{\nu,\theta}}
\newcommand{\Ol}{{\mathcal O}_{\nu}}
\newcommand{\CR}{{\mathcal R}}
\newcommand{\CH}{{\mathcal H}}

\newcommand{\sk}{{\mathsf k}}

\renewcommand{\a}{\alpha}
\renewcommand{\b}{{\beta}}
\newcommand{\g}{\gamma}
\renewcommand{\d}{\delta}
\renewcommand{\l}{\lambda}
\renewcommand{\O}{{\Omega}}
\newcommand{\s}{\sigma}
\renewcommand{\S}{\Sigma}
\renewcommand{\1}[1]{F_{1,{#1}}}
\renewcommand{\2}[1]{F_{2,{#1}}}

\newcommand{\tw}{{\tilde w}}
\newcommand{\St}{\mathrm{St}}
\newcommand{\bo}{\bigoplus}
\newcommand{\btu}{\bigtriangleup}
\newcommand{\invl}{\leftrightarrow}

\newcommand{\ol}{\overline}
\newcommand{\Ql}{{{\underline{\overline{\mathbb Q}}}{}_l}}

\newcommand{\utH}{\tilde{\underline H}{}}

\newcommand {\uf}{\times_{{}_{\Fl(V)}}}
\newcommand{\zowt}[1]{\overline{\O}_{\tilde{w}_{#1}}}
\newcommand{\zowts}{\overline{\O}_{\tilde{w}*s}}
\newcommand{\owt}[1]{\O_{\tilde{w}_{#1}}}
\newcommand{\zos}[1]{\overline{\O}_{s_{#1}} }
\newcommand{\owts}[1]{\O_{\tilde{w} s_{#1}} }
\newcommand{\os}[1]{\O_{s_{#1}}}
\newcommand{\ot}[1]{\O_{\tilde{#1}} }
\let\x\times

\newcommand{\tT}{\tilde{T}}
\newcommand{\alg}[1]{{#1}\otimes_{\BZ[\bv,\bv^{-1}]}\QQ(\bv) }
\newcommand{\Z}[1]{\Fl_{#1}}
\renewcommand{\1}[1]{F_{1,{#1}}}
\renewcommand{\2}[1]{F_{2,{#1}}}

\newcommand\ncmd\newcommand
\ncmd{\young}[1]{\ensuremath{\vcenter{%
\def\byng{\begin{Young}}
\byng
#1
\crcr\end{Young}}}} 
\NewEnviron{yt}{\young{\BODY}}
  \ncmd\ycr\cr

\ncmd\setm\setminus

\begin{document}
\title{Mirabolic Robinson-Schensted-Knuth correspondence}
\author{Roman Travkin}
\address{Massachusetts Institute of Technology\\
Mathematics Department, Cambridge MA 02139 USA}
\email{travkin@alum.mit.edu}

\begin{abstract}
The set of orbits of $GL(V)$ in $Fl(V)\times Fl(V)\times V$ is finite,
and is parametrized by the set of certain decorated permutations in a
work of Solomon. We describe a Mirabolic RSK correspondence
(bijective) between this set of decorated permutations and the set of triples:
a pair of standard Young tableaux, and an extra partition. It gives rise to
a partition of the set of orbits into combinatorial cells. We prove that the
same partition is given by the type of a general conormal vector to an orbit.
We conjecture that the same partition is given by the bimodule Kazhdan-Lusztig
cells in the bimodule over the Iwahori-Hecke algebra of $GL(V)$ arising from
$Fl(V)\times Fl(V)\times V$. We also give conjectural applications to the
classification of unipotent mirabolic character sheaves on $GL(V)\times V$.
\end{abstract}

\maketitle

\section{Introduction}

\subsection{}
Let $v\in V$ be a nonzero vector in an $N$-dimensional vector space over
a field $\sk$. The stabilizer $P_N$ of $v$ in $\GL_N=\GL(V)$ is called a
{\em mirabolic} subgroup of $\GL_N$. The special properties of the pair
$P_N\subset\GL_N$ are among the principal reasons why the
representation theory
of $\GL_N$ is in many respects simpler than that of the other reductive
groups over $\sk$ (see e.g.~\cite{B},~\cite{G}).
One more remarkable feature of the pair $P_N\subset\GL_N$ was discovered by
P.~Etingof and V.~Ginzburg a few years ago. Namely, the quantum Hamiltonian
reduction of the differential operators on $\GL_N$ with respect to $P_N$ is
isomorphic to the spherical trigonometric Cherednik algebra $H_N$
(see e.g.~\cite{EG}); equivalently, the quantum Hamiltonian reduction of
the differential operators on $\GL_N\times V$ with respect to $\GL_N$ is
isomorphic to $H_N$. Thus one is led to study the D-modules on
$\GL_N\times V$ whose quantum Hamiltonian reduction lies in the category
${\mathcal O}$ for $H_N$ (see~\cite{FG}).
The corresponding perverse sheaves are called {\em mirabolic character
sheaves}; they are close relatives of Lusztig's character
sheaves (see e.g.~\cite{L}).
The present work
is a first step towards a classification of mirabolic character sheaves.

\subsection{}
According to Lusztig's classification of character sheaves, the set of
isomorphism classes of unipotent character sheaves on a reductive group $G$
is partitioned into cells, which correspond bijectively to special unipotent
classes in $G$. For $G=\GL_N$, each unipotent class is special, and each cell
contains a unique character sheaf; thus the unipotent character sheaves are
classified by their (nilpotent) singular supports, so they are numbered by
partitions of $N$.

Finally, recall that the cells in question are
the two-sided Kazhdan-Lusztig cells of the finite Hecke algebra
${\mathcal H}_N$. If $\Fl(V)$ stands for the flag variety of $\GL(V)$, then
${\mathcal H}_N$ is the Grothendieck ring of the constructible
$\GL(V)$-equivariant mixed Tate complexes on $\Fl(V)\times\Fl(V)$
(multiplication given by convolution). The two-sided cells arise from
the two-sided ideals spanned by the subsets of the Kazhdan-Lusztig basis
(formed by the classes of Goresky-MacPherson sheaves). This basis is
numbered by the symmetric group $\gS_N$, and its partition into two-sided
cells is given by the Robinson-Shensted-Knuth algorithm, see~\cite{KL}.
A $\GL(V)$-orbit in $\Fl(V)\times\Fl(V)$ numbered by $w\in\gS_N$ lies in
a two-sided cell $\lambda$ iff a general conormal vector to the orbit is
a nilpotent element of type $\lambda$, see~\cite{S}.

\subsection{}
The starting point of our work is a classification of $\GL(V)$-orbits in
${\mathcal N}\times V$ where $\mathcal N$ is the nilpotent cone in $\End(V)$
(it was independently obtained by P.~Achar and A.~Henderson in~\cite{AH}).
We prove (see section~\ref{comparison}) that the set of orbits is in a natural
bijection with the set $\mathfrak P$ of pairs of partitions $(\nu,\theta)$
such that $|\nu|=N$, and $\nu\supset\theta$, that is
$\nu_i\geq\theta_i\geq\nu_{i+1}$ for any $i\geq1$. Note that $\mathfrak P$
arises also in Zelevinsky's classification of restrictions of unipotent
irreducible representations of $\GL_N({\mathbb F}_q)$ to $P_N({\mathbb F}_q)$
(see~\cite{Z}, Theorem~13.5), and this coincidence is not accidental.

A conormal vector to a $\GL(V)$-orbit in $\Fl(V)\times\Fl(V)\times V$ lies in
the variety $Z$ of quadruples $(u_1,u_2,v,v^*)$ where $v\in V$, and
$v^*\in V^*$, and $u_1,u_2$ are nilpotent endomorphisms of $V$ such that
$u_1+u_2+v\otimes v^*=0$. The set of orbits of $\GL(V)$ in $Z$ is infinite,
and $Z$ is reducible (it has $N+1$ irreducible components of dimension $N^2$)
but we define in~\ref{XYZ} a collection of closed irreducible subvarieties of
$Z$ numbered by the triples $(\nu\supset\theta\subset\nu')$ of partitions
such that $|\nu|=|\nu'|=N$. These subvarieties are the images of the closures
of the conormal bundles to $\GL(V)$-orbits in $\Fl(V)\times\Fl(V)\times V$;
they are mirabolic analogues of
the nilpotent orbit closures in $\mathcal N$.

The Hecke algebra ${\mathcal H}_N$ acts by the right and left convolution
on the Grothendieck group of the constructible $\GL(V)$-equivariant mixed
Tate complexes on $\Fl(V)\times\Fl(V)\times V$; we will denote this
bimodule by ${\mathcal R}_N$. It comes equipped with a Kazhdan-Lusztig
basis numbered by the finite set $RB_N$ of $\GL(V)$-orbits in
$\Fl(V)\times\Fl(V)\times V$, described in~\cite{Sol} (see also~\cite{M,MWZ}).
Thus we can define a partition of $RB_N$ into {\em bimodule KL cells}.
In this note we define an analogue of the RSK algorithm which is conjectured
to be connected with these bimodule cells.
Our {\em mirabolic RSK correspondence} (see subsection~\ref{RSK2})
is a bijection between the set $RB_N$ of
colored permutations of $\{1,\ldots,N\}$, and the set of triples
$\{(T_1,T_2,\theta)\}$ where $T_1$ (resp. $T_2$) is a standard tableau of
the shape $\nu$ (resp. $\nu'$) where $|\nu|=|\nu'|=N$, and $\theta$ is
another partition such that $\nu\supset\theta\subset\nu'$.

We conjecture that the colored permutations $\tw,\tw'\in RB_N$ lie in
the same bimodule KL cell iff the output of the mirabolic RSK algorithm
on $\tw,\tw'$ gives the same partitions: $\nu(\tw)=\nu(\tw'),\
\nu'(\tw)=\nu'(\tw'),\ \theta(\tw)=\theta(\tw')$ (see Theorem~\ref{L}
for a partial result in this direction). An equivalent form of the conjecture
states that the ${\mathcal H}_N$-subquotient bimodules of ${\mathcal R}_N$
supported by the bimodule KL cells are irreducible
(cf. Proposition~\ref{algH}).
We also define a partition of $RB_N$ into {\em microlocal two-sided cells}
according to the type of a general conormal vector to the corresponding orbit.
We prove that the colored permutations $\tw,\tw'\in RB_N$ lie in the same
microlocal two-sided cell iff the output of the mirabolic RSK algorithm
on $\tw,\tw'$ gives the same partitions $\nu\supset\theta\subset\nu'$
(see Theorem~\ref{fthm}). In subsection~\ref{FD} we describe combinatorially
the involution $\operatorname F$ on $RB_N$ arising from the Fourier-Deligne
transform from the category of $\GL(V)$-equivariant sheaves on
$\Fl(V)\times\Fl(V)\times V$ to the category of $\GL(V)$-equivariant sheaves
on $\Fl(V^*)\times\Fl(V^*)\times V^*$.
In subsection~\ref{char} we give a conjectural application to the
classification of unipotent mirabolic character sheaves.
In subsection~\ref{asymp} we formulate a conjecture on the structure of the
asymptotic bimodule over Lusztig's asymptotic ring $J$ for {\em diagonal}
bimodule KL cells: those corresponding to triples $\nu\supset\theta\subset\nu$
(that is, $\nu=\nu')$.

\subsection{}
Let us emphasize that almost all arguments and constructions in the paper
are of elementary combinatorial and linear algebraic nature. For instance,
even though the bimodule ${\mathcal R}_N$ over the Hecke algebra
${\mathcal H}_N$ is of geometric origin, it is described explicitly in
Propositions~\ref{thm_expl} and~\ref{thm_explicit}. The Kazhdan-Lusztig basis
of ${\mathcal R}_N$ is defined by an inductive combinatorial algorithm,
similarly to the Kazhdan-Lusztig basis of ${\mathcal H}_N$.
Only the description of the $W$-graph of the ${\mathcal H}_N$-bimodule
${\mathcal R}_N$ in Proposition~\ref{HKL} does rely on geometric
considerations.

\subsection{Acknowledgments}
I am grateful to P.~Achar and A.~Henderson for sending me~\cite{AH}
prior to its publication, where our Theorem~\ref{Thm1} is proved independently
(as Proposition~2.3). I am indebted to M.~Finkelberg for posing the problem,
numerous valuable discussions and help in editing the paper.  I would also like 
to thank G.~Lusztig for pointing out the reference~\cite{Sol} which also classifies the 
orbits in $\Fl(V)\x\Fl(V)\x V$ and studies $\mathcal R$ as an algebra,
 and D.~Zakharov for finding a missing case in the list of conditions in the 
formulas for the actions of the Hecke algebra on the mirabolic bimodule 
(see Propositions \ref{thm_expl} and \ref{thm_explicit}) in a previous version of the paper.
I am thankful to Independent University of Moscow for education,
financial support and
various help. I thank P.~Etingof for creating the ideal conditions for my work.

\section{$\GL(V)$-orbits in ${\mathcal N}\times V$}

\subsection{}
The following Theorem essentially goes back to J.~Bernstein, who proved
in~\cite{B},~section~4.2, the finiteness of the set of
$P_N$-orbits in the nilpotent cone of $\gl_N$. It was independently
proved by P.~Achar and A.~Henderson (\cite{AH}, Proposition~2.3).

\begin{thm}
\label{Thm1}
Let ${\mathcal N}\subset \gl(V) $ be the nilpotent cone.
There is a one-to-one correspondence between $\GL(V)$-orbits in
${\mathcal N} \times V $ and pairs of partitions $(\lambda,\mu)$
such that $|\lambda |+|\mu |= \sum \lambda_i+\sum \mu_i=N$.
Furthermore, if a pair $(u,v)\in{\mathcal N}\times V$ belongs
to the orbit corresponding to the pair $(\lambda,\mu)$ then the type of
$u$ is equal to $\lambda+ \mu=(\lambda_1+ \mu_1,\lambda_2+ \mu_2,\dots)$ .
\end{thm}


\begin{proof}
Given a pair $(\lambda,\mu)$ such that $|\lambda|+|\mu|=N$, we will
construct the pair $(u,v)$ in the following way. Let $\nu=\lambda+\mu$
and $u$ be a nilpotent of type $\nu$. Denote by $D_\nu$ the set of boxes
of the Young diagram $\nu$, i.e. $D_\nu=\{(i,j)\mid 1\le j\le\nu_i\}$. Choose
a basis $e_{i,j}$     $((i,j)\in D_\nu)$ such that $ue_{i,j}=e_{i,j-1}$ for
$2\le j\le\nu_i$ and $ue_{i,1}=0$. Let $v=\sum\limits_i e_{i,\lambda_i}$
where we put $e_{i,0}=0$.

The inverse correspondence is obtained as follows.
Let $(u,v)\in\N\times V$. Denote by $Z(u)$ the centralizer of $u$ in the
algebra $\End(V)$.
Let $\nu$ be the type of $u$ and $\lambda$ be the type of $u|_{Z(u)v}$
and $\mu$  be the type of $u|_{V/Z(u)v}$.

Let us prove that these two correspondences are mutually inverse.
We will need the following lemma.

\begin{lem}
\label {0}
Let $A\subset \End(V) $ be an associative algebra with identity and $A^\times$
the multiplicative group of $A$. Suppose the $A$-module $V$ has finitely many
submodules. Then $A^\times$-orbits in $V$ are in one-to-one correspondence with
these submodules. Namely, each $A^\times$-orbit has the form $\Omega_S :=
S\setminus\bigcup\limits_{\text{\normalfont submodules $S'\varsubsetneq
S$}}\hspace{-1cm}S'$  \phantom{wh} where $S$ is an $A^\times$-submodule of $V$.
	
\end{lem}

\begin{proof}
It is clear that the sets $\Omega_S$ give us a decomposition of $V$
into a union of locally closed
subvarieties. So, we must prove that two points $v,v'\in V$ belong to the same
$A^\times$-orbit iff
they belong to  the same $\Omega_S$, i.~e. they generate the same $A$-submodule
$S=Av=Av'$.
If $v$ and $v'$ belong to the same $A^\times$-orbit then $v'=av$ for some
$a\in A^\times$ and $Av=Aav=Av'$.
Conversely, let $v,v'\in\Omega_S$ for some $S$, so that  $Av=Av'=S$.
Now $A^\times$ is a non-empty Zariski open, and hence dense, subset in~$A$.
Consequently, we see that $A^\times v$  and $A^\times v'$ are
constructible dense subsets of $S$. This implies that
$A^\times v \cap A^\times v'\ne\varnothing$  and  therefore   $A^\times
v=A^\times v'$. \end{proof}

\subsubsection{}
Let us deduce the theorem from the lemma. Fix a partition $\nu$ of $N$.
Consider all the $\GL(V)$-orbits in $\N\times V$ consisting of points
$(u,v)$ where $u$ has the type $\nu$. These orbits correspond to
$\GL(V)_u$-orbits in $V$ where $\GL(V)_u$ is the stabilizer of $u$ in
$\GL(V)$. Note that $\GL(V)_u=(Z(u))^\times$.
According to the lemma it suffices to prove that $V$ has finitely many
$Z(u)$-submodules and find all these submodules. Consider $V$ as a
$\sk[t]$-module where $t$ acts by $u$. This module is isomorphic to
$\bigoplus\limits_i \sk[t]/(t^{\nu_i}\sk[t])$. Let
$V_i\cong \sk[t]/(t^{\nu_i}\sk[t])\subset V$ be the the $i$-th
direct summand of
this sum. For each $i$  let $\{e_{i,j}\}_{j=1}^{\nu_i}$ be a basis
of $V_i$ such
that $ue_{i,j}= e_{i,j-1}\ (j\ge2)$ and $ue_{i,1}=0$. We  can write
$$Z(u)=\End_{\sk[t]}(V)=\bigoplus_{i,i'}\Hom_{\sk[t]}(V_i,V_{i'})\cong
\bigoplus_{i,i'}\sk[t]/(t^{\min\{\nu_i,\nu_{i'}\}}\sk[t])$$
Let $a_{i,i'}$ be a generator of the $\sk[t]$-module
$\Hom_{\sk[t]}(V_i,V_{i'})\subset \End_{\sk[t]}(V)$ given by
$$a_{i,i'}e_{i_1,j}=\delta_{i,i_1}e_{i',(j-\max\{0,\nu_i-\nu_{i'}\})}
\qquad\qquad (\text{we put  } e_{i,j}=0 \text{  for $j\le 0$}).$$
 Now let $S$ be a $Z(u)$-submodule of $V$.
Sinse $S$ is invariant under  $a_{i,i'}$ for all $i$,
$S$ has a form $S=\bigoplus_i S_i$ where $S_i\subset V_i$.
Further since $S$ invariant under $u\in Z(u)$ all the $S_i$ have the
form $u^{\mu_i}V_i$. Put $\lambda_i=\nu_i-\mu_i$. The invariance of $S$
under all $a_{i,i'}$ is equivalent to the fact that $\lambda$ and $\mu$
are partitions, i.e. $\lambda_1\ge\lambda_2\ge\dots$ and
$\mu_1\ge\mu_2\ge\dots$. So we have shown that $Z(u)$-submodules of $V$
are in one-to-one correspondence with pairs of partitions $(\lambda, \mu)$
such that $\lambda+\mu=\nu$. An application of Lemma~\ref{0} concludes the
proof of the theorem. \end{proof}




\subsection{Comparison with Zelevinsky's parametrization}
\label{comparison}
A.~Zelevinsky considers in~\cite{Z},~Theorem~13.5.a)
the set $\mathfrak Z$ of isomorphism
classes of pairs $(U,W)$ where $U$ is an irreducible unipotent complex
representation of the finite group $\GL_N({\mathbb F}_q)$, and $W$ is an
irreducible constituent of the restriction of $U$ to the mirabolic subgroup
$P_N({\mathbb F}_q)$. He constructs a natural bijection between $\mathfrak Z$
and the set $\mathfrak P$ of pairs of partitions $(\nu,\theta)$ such that
$|\nu|=N$, and $\tilde \nu_j-1\le\widetilde{\theta_j}\le\tilde\nu_j$
for all $j$
(this is equivalent to $\nu_i\ge\theta_i\ge\nu_{i+1}$ for all $i$).
The following Proposition-Construction establishes a natural bijection
between $\mathfrak P$, and the set of pairs of partitions $(\lambda,\mu)$
such that $|\lambda|+|\mu|=N$.

\begin{propconstr}
\label{Thm2}
Let $\nu$ be a partition and $\tilde{\nu}$
the conjugate partition so that $\nu_i\ge j \iff\tilde\nu_j\ge i$.
There exists a natural one-to-one correspondence between pairs
$(\lambda,\mu)$ of partitions such that
$\lambda+\mu=\nu$, and partitions $\theta$ such
that $\tilde \nu_j-1\le\widetilde{\theta_j}\le\tilde\nu_j$ for all $j$
(this is equivalent to $\nu_i\ge\theta_i\ge\nu_{i+1}$ for all $i$).
This correspondence is given by
\begin{gather}
\label{pc1}\theta_i=\lambda_{i+1}+\mu_i;\\
\label{pc2}\left\{  \begin{aligned}
\lambda_i&=\sum_{k=i}^\infty (\nu_k-\theta_k)=
\nu_i-\theta_i+\nu_{i+1}-\dots,\\
\mu_i&=\sum_{k=i}^\infty (\theta_k-\nu_{k+1})=
\theta_i-\nu_{i+1}+\theta_{i+1}-\dots.
\end{aligned}\right.
\end{gather}
\end{propconstr}

\begin{proof} It is easy to see that equations~\eqref{pc1}
and~\eqref{pc2} give mutually inverse correspondences. \end{proof}

We will denote the above correspondence by
$(\nu,\theta)=\Upsilon(\lambda,\mu),\ (\lambda,\mu)=\Xi(\nu,\theta)$.

\begin{cor}
\label{hor}
A pair $(u,v)$ lies in an orbit $({\mathcal N}\times
V)_{(\lambda,\mu)}$ such that $(\lambda,\mu)=\Xi(\nu,\theta)$ iff the
Jordan type of $u$ is $\nu$, and the Jordan type of $u|_{V/\langle
  v,uv,u^2v,\ldots\rangle}$ is $\theta$.
\end{cor}

\noindent{\it Proof\/} is obvious from the construction. \qed

\section{$\GL(V)$-orbits in $\Fl(V)\times\Fl(V)\times V$}

\subsection{}
Let $(F_1,F_2,v)\in\Fl(V)\times\Fl(V)\times V$. Consider the orbit
$\GL(V)\cdot (F_1,F_2,v)$. If $v=0$ then this orbit lies in
$\Fl(V)\times\Fl(V)\times\{0\}$. Such orbits can be parametrized by
permutations of $N$ elements. Otherwise, if $v\ne 0$ the orbit is preimage
of an orbit in $\Fl(V)\times\Fl(V)\times\PP(V)$. This follows from the fact
that if $c\in \sk^\times$ then the element $(F_1,F_2,cv)$ can be
obtained from    $(F_1,F_2,v)$ by the action of the scalar operator
$c \cdot \mathrm{id}\in \GL(V)$. Such orbits are in one-to-one correspondence
with pairs $(w,\sigma)$ where $w\in \gS_N$ is a permutation and $\sigma$ is
non-empty, decreasing  subsequence of $w$ (see~\cite{Sol,MWZ}).
So $\GL(V)$-orbits in
$\Fl(V)\times\Fl(V)\times V$ are indexed by pairs $(w,\sigma)$ where
$w\in \gS_N$ and $\sigma$ is a decreasing subsequence of $w$ (possibly empty).
We will give another proof of this fact in the following lemma.

\begin{lem}
\label{1}
 There is a one-to-one correspondence between $\GL(V)$-orbits in
$\Fl(V)\times\Fl(V)\times V$ and pairs $(w,\sigma)$ where
$w\in \gS_N$ and $\sigma\subset\{1,2,\dots,N\}$ such that if
$i,j\in\sigma$ and $i<j$ then $w(i)>w(j)$. These orbits can be also
indexed by pairs $(w,\beta)$ where $\beta\subset\{1,2,\dots,N\}$ is
a subset such that if
$i\in \{1,\dots,N\}\setminus\beta$ and $j\in\beta$ then either $i>j$
or $w(i)>w(j)$.
\end{lem}

We denote by $RB$ the set of such pairs $(w,\beta)$.
We think of elements of $RB$ as of words colored in two colors: red and blue.
Namely, if $(w,\beta)\in RB$ we consider the word $w(1)\dots w(N)$ and
paint $w(i)$ in blue if $i\in \beta$, and we paint it in
red if $i\not\in \beta$.

\bigskip

\begin{proof}
For each $w\in\gS_N$ let $\Omega_w$ be the corresponding $\GL(V)$-orbit in
$\Fl(V)\times\Fl(V)$. Namely,
$(F_1,F_2)\in\Omega_w$ iff there exists a basis $\{e_i\}$ of $V$ such that
\begin{gather}   \label{F1} F_{1,i}=\langle e_1,\dots,e_i\rangle\\
               \label{F2} F_{2,j}=\langle e_{w^{-1} (1)},\dots,e_{w^{-1} (j)}\rangle
\end{gather}
Consider all the $\GL(V)$-orbits in $\Fl(V)\times\Fl(V)\times V$ consisting of
such points $(F_1,F_2,v)$ that
$(F_1,F_2)\in\Omega_w$ where $w$ is fixed.  Fix a pair $(F_1,F_2)\in\Omega_w$
and let $H$ be its
stabilizer in $\GL(V)$. Then these orbits correspond to $H$-orbits in $V$. Let
$A_k\subset\End(V)\ (k=1,2)$ be
the subalgebra defined by
$$a\in A_k \iff  \forall i\quad a(F_{k,i})\subset F_{k,i}. $$
Denote $A=A_1\cap A_2$. Then $H=A^\times$ and we can apply Lemma~\ref{0}.

 Let $\{e_i\}$ be a basis satisfying~\eqref{F1}   and~\eqref{F2} and   $E_{i,j}$
the operator given by
\eql{Eij}{E_{i,j}e_{j'}=\delta_{j,j'}e_i.} Then $$A=\bigoplus_{\substack{i\le i'\\w(i)\le
w(i')}} \sk E_{i,i'}.$$ Now it is easy to see
that  all the $A$-submodules in $V$ have the form $S(\beta):=
\bigoplus_{i\in\beta} \sk e_i$ where $\beta$  satisfies
the condition of the lemma. So, applying Lemma~\ref{0} proves the second part of
the lemma. We will denote by
$\Omega_{w,\beta}$ the orbit in $\Fl(V)\times\Fl(V)\times V$ corresponding to
$(w,\beta)$.

For each $(w,\beta)\in RB$ let
\eql{sigma}{
\s= \s(\tw)= \{i\in\beta\mid \forall j\ (j>i)\,\,\&\,{(w(j)> w(i))} \Longrightarrow
j\not\in\beta\}.}
It is easy to see that $(w,\beta)$ and $(w,\sigma)$
can be reconstructed from each other.
So the lemma is proved. \end{proof}

Note that $\Omega_{w,\beta}$ consists of such triples $(F_1,F_2,v)$ that there
exists a basis $\{e_i\}$ satisfying~\eqref{F1},~\eqref{F2} and
such that\footnote{This formula is different from the one in~\cite{MWZ}:
$ v=\sum_{i\in\s}e_i.$}
\begin{equation}\label{v} v=\sum_{i\in\b}e_i. \end{equation}

\subsection{$X,Y,Z$ and two-sided microlocal cells}
\label{XYZ}
We denote $\Fl(V)\times\Fl(V)\times V$ by $X$,
and consider the cotangent bundle $T^*X$. It can be  described as the
variety of sextuples $(F_1,F_2,v,u_1,u_2,v^*)\in T^*(X)$
where $(F_1,F_2,v)\in X$, $u_i$ $(i=1,2)$ are nilpotent operators on $V$,
$u_i$ preserves $F_i$  and $v^*\in V^*$. The moment map
$T^*X\to\gl(V)^*\cong\gl(V)$ sends a point $(F_1,F_2,v,u_1,u_2,v^*)\in T^*(X)$
to the sum
$u_1+u_2+v\otimes v^*$. The preimage  $Y$ of $0$ under
this map is the union of conormal bundles of
$\GL(V)$-orbits in $X$. So all the irreducible components of $Y$
have the form $Y_{w,\sigma}=\overline{N^*\Omega_{w,\sigma}}$.
We determine the type of $u_1$ for a general point of $N^*\Omega_{w,\sigma}$.

Now consider the projection $\pi\colon Y\to\Fl(V)\times
V\times\N\times\N\times V^*$,
$(F_1,F_2,v,u_1,u_2,v^*)\longmapsto(F_2,v,u_1,u_2,v^*)$. Let
$\tilde Y=\pi(Y)$. The preimage of a point $(F_2,v,u_1,u_2,v^*)\in
\tilde Y$ is isomorphic to the variety $\Fl_{u_1}(V)$ of full flags
 fixed by $u_1$. This variety is known to be pure-dimensional and the set of
its irreducible components can be identified with the set $\St(\l)$ of
standard tableaux of the shape $\lambda$ where $\lambda$ is the type of
$u_1$. Namely, for each $T\in\St(\l)$ the corresponding irreducible
component $\Fl_{u_1,T}$ of $\Fl_{u_1}(V)$ is defined as follows.
Let $\l^{(i)}(T)$ be the shape of the subtableau of $T$ formed by numbers
$1,\dots,i$. Then $\Fl_{u_1,T}$ is the closure of the set $\Fl_{u_1}^T$ of
all $F\in\Fl_{u_1}(V)$ such that $u_1|_{F_i}$ has the type $\l^{(i)}(T)$.

Let $Z$ be the variety of quadruples
 $(u_1,u_2,v,v^*)$, where  $(u_1,u_2)\in \N$, $v\in V$, $v^*\in V^*$ and
$u_1+u_2+v\otimes v^*=0$.
Then we have a projection $\pi:Y\to Z$. We say that $\tw, \tw'\in RB$
belong to the same two-sided microlocal cell if
$\pi(Y_{\tw})=\pi(Y_{\tw'})$. We denote by $\fP$ the set of
pairs of partitions $(\nu,\theta)$ such that
$|\nu|=N$ and $\nu_i\ge\theta_i\ge\nu_{i+1}$ for each $i\ge 1$.
Further, denote by $\bT$ the set of triples
of partitions $(\nu,\theta,\nu')$ such that $(\nu,\theta)\in\fP$
and $(\nu',\theta)\in \fP$. For any $\bt=(\nu,\theta,\nu')\in\bT$
denote by $Z^\bt$ the set of quadruples $(u_1,u_2,v,v^*)\in Z$
such that the types of $u_1,u_2$ and $u_1|_{V/k[u_1]v}$
are equal to $\nu,\nu'$ and $\theta$
respectively (it is easy to check that for each quadruple
$(u_1,u_2,v,v^*)\in Z$ we have $\sk[u_1]v=\sk[u_2]v$ and
 $u_1|_{V/\sk[u_1]v}= u_2|_{V/\sk[u_1]v}$).

\subsection{}
\label{mir}
We fix $\tw\in RB$. Let $y$ be a general point of variety
$Y_{\tw}={\overline {N^*(\owt{})}}$.
We take $\bt=\bt(\tw)=(\nu,\theta,\nu')\in\bT$
such that $\pi(y)\in Z^{\bt}$. We consider the standard Young tableaux
$T_1=T_1(\tw)\in\St(\nu)$ and
$T_2=T_2(\tw)\in\St(\nu')$ such that $F_i(y)\in\Fl_{u_i,T_i}$   $(i=1,2).$

\begin{prop}\label{N}
The map $\tw\mapsto(\bt(\tw),T_1(\tw),T_2(\tw))$ realizes a
one-to-one correspondence between $RB$ and the set of triples
$(\bt,T_1,T_2)$ such that $\bt=(\nu,\theta,\nu')\in\bT$,
$T_1\in\St(\nu), T_2\in\St(\nu')$. Moreover $\tw$
and $\tw'$ belong to the same two-sided
microlocal cell iff $\bt(\tw)=\bt(\tw')$.
\end {prop}

\begin{proof}

Denote by $Y^{\bt,T_1,T_2}$ the set of points $y\in Y$
such that $\pi(y)\in Z^\bt$
and $F_i(y)\in\Fl_{u_i,T_i}(V)$. These sets are locally closed,
disjoint, and $Y$ is their union.  We claim that  all of them are
open subsets of irreducible components of~$Y$. We will use the
formula~\eqref{f3} (see page~\pageref{f3}) whose proof does not use
the proposition we are proving. (See also Remark~\ref{alt} below.)
Note that the number of the sets $Y^{\bt,T_1,T_2}$
coincides with the number of irreducible components of $Y$.
This follows from the fact that the number of these sets is equal
to the rank of the right hand side of
the formula~\eqref{f3}, and this rank coincides with the cardinality
of $RB$, i.e. with the number  of irreducible components  of $Y$.
Therefore, if all these sets are irreducible then their closures must
be irreducible components of $U$. In this case  we obtain a bijection
of required form. Hence it is enough to prove that the sets
$Y^{\bt,T_1,T_2}$ are irreducible. Note that all the fibers
of the projection $Y^{\bt,Y_1,T_2}\to Z^\bt$ have the form
$\Fl_{u_1,T_1(\tw)}\times\Fl_{u_2,T_2(\tw)}$. It means they are irreducible
and have the same dimension. So it is enough to prove that $Z^\bt$
is irreducible.

 Let $\bt=(\nu,\theta,\nu')$. Let $\Olm$ be an orbit in $\N\times\ V$
corresponding to the pair $(\nu,\theta)$, i.e.  the set of all
$(u,v)\in \N\times V$ such that the
type  of $u$ is equal to $\nu$ and the type of
$u|_{V/(\sk[u]\cdot v)}$ is equal
to~$\theta$. We have the natural pojection $Z^{\bt}\to\Olm$.
The fiber of this map over
a point $(u,v)$ is isomorfic to the set of $v^*\in V^*$ such that
$u+v\otimes v^*\in\Ol$ where $\Ol\subset\N$. One can  check that this subset
is an open subset of an affine subspace of $V^*$. So the fibers of
this projection are irreducible. Besides, this bundle is homogeneous.
Since the orbit $\Olm$ is irreducible,
we obtain that $Z^\bt$ is irreducible.
\end{proof}

\begin{Rem}\label{alt} Instead of using the formula~\eqref{f3},
one can directly
compute the dimension of the sets $Y^{\bt,T_1,T_2}$, showing that
$\dim Y^{\bt,T_1,T_2} =\dim Y$, which amounts to proving the equation
$$
  \dim Z^\bt = N^2-n(\nu)-n(\nu')
$$
where $\bt=(\nu,\theta,\nu')$, and $n(\nu)=\sum_{i\ge1} (i-1)\nu_i$.
\end{Rem}

\subsection{Notation}\label{RSKmir}
We will call the map $\tw\mapsto(\bt(\tw),T_1(\tw),T_2(\tw))$
constructed in~\ref{mir} {\em the mirabolic RSK correspondence}
and denote it by $\RSK$.

\subsection{The description of mirabolic RSK correspondence}\label{RSK2}
We are going to give a combinatorial description of mirabolic
RSK correspondence defined in Proposition~\ref{N}.
Let $\tw=(w,\b)\in RB$. We will construct step by step a
standard Young tableau. Besides we will need a separate
row of infinite length (denote it by $r^@$) consisting originally from
the symbols ``@''. We assume that
``@'' is greater than all the numbers from $1$ to $N$.

We will run next procedure successively for $i=1,2,\ldots,N$ :

1a. If $i\in\b $ then insert $w(i)$ into the tableau $T^@$ (originally empty)
according to the standard row bumping rule of
the RSK algorithm described in~\cite{Fulton} ( The tableau $T^@$ changes as
the next element is inserted).

1b. If $i\notin\b $ then insert first $w(i)$ into $r^@$ instead of the least
element greater than $w(i)$,
and then insert the element removed from $r^@$ by replacing into tableau $T^@$
via row bumping algorithm~(see~\cite{Fulton}.)

2. After all the elements $w(1),\ldots,w(N)$ are inserted, we should
insert the elements of $r^@$ successively via
standard row bumping algorithm.

3a. After that we construct $T_2(\tw)$ from the tableau $T^@$ by
throwing out all the symbols ``@''.

3b. $T_1(\tw)$ is defined as the standard tableau where number
$i$, $1\le i\le N$ stands in the box that was added into $T^@$ at the $i$-th step 1a--1b.
(No boxes created in step~2 are included in $T_2(\tw)$.)

3c. Finally, $\bt(\tw)=(\nu,\theta,\nu')$ where $\nu=\Sh(T_1(\tw))$;
\quad $\nu'=\Sh(T_2(\tw))$;\quad $\theta=(\Sh(T^@_*))_-$ \quad and we have
denoted
\begin{itemize}

\item    $\Sh$ -- the operation of taking the shape of a tableau;

\item    $()_-$ --  the operation of removing of the first part of a partition;

\item    $T^@_*$ -- the tableau $T^@$ obtained at the last step of the algorithm.

\end{itemize}

\bigskip

Let us illustrate the above construction by the following example.

\subsection{An example}\label{example}
Let $N= 10$,               $w={7,2,5,1,6,9,4,8,10,3}$;
$\b=\{1,2,3,4,7\}$.
The tableaux $T^@$ and the row $r^@$ obtained at the $i$-th step of the algorithm 
will be denoted by  $T^@_i$ and  $r^@_i$, respectively.  The steps of the mirabolic
RSK algorithm are shown on Fig.~\ref{fig:RSK}.

\begin{figure}
\parbox{0.9\textwidth}{\noindent\flushleft
\begin{tabular}{rl>{$}p{11em}<{$}@{\qquad\qquad}lllllll}
1.& $T^@_1=$ & \young{7\cr} & $r^@_1=\young{@&@&{\ldots}} $  \\[\medskipamount]
2.& $T^@_2=$ & \young{2\cr 7 \cr }
& $r^@_2=\young{@&@&{\ldots}} $  \\[.7cm]
3.& $T^@_3=$ & \young{2&5\cr7  \cr }
& $r^@_3=\young{@&@&{\ldots}} $  \\[.7cm]
4.& $T^@_4=$ & \young{1&5\cr2\cr7  \cr }
& $r^@_4=\young{@&@&{\ldots}} $  \\[1.2cm]
5.& $T^@_5=$ & \young{1&5&@\cr2\cr7  \cr }
& $r^@_5=\young{6&@&@&{\ldots}} $  \\[1.2cm]
6.& $T^@_6=$ & \young{1&5&@&@\cr2\cr7  \cr }
& $r^@_6=\young{6&9&@&@&{\ldots}} $  \\[1.2cm]
7.& $T^@_7=$ & \young{1&4&@&@\cr2&5\cr7  \cr }
& $r^@_7=\young{6&9&@&@&{\ldots}} $  \\[1.2cm]
8.& $T^@_8=$ & \young{1&4&9&@\cr2&5&@\cr7  \cr }
& $r^@_8=\young{6&8&@&@&{\ldots}} $  \\[1.2cm]
9.& $T^@_9=$ & \young{1&4&9&@&@\cr2&5&@\cr7  \cr }
& $r^@_9=\young{6&8&{10}&@&@&{\ldots}} $  \\[1.2cm]
10.& $T^@_{10}=$ & \young{1&4&6&@&@\cr2&5&9\cr7&@
\cr } & $r^@_{10}=\young{3&8&{10}&@&@&{\ldots}} $
\end{tabular}

\ncmd\syoung{\rightsquigarrow\young}

\nopagebreak\vspace{.3cm}
\noindent
\begin{tabular}{rl@{\quad}l@{\quad}l@{\quad}ll}
11.&$ \young{   1&3&6&@&@\cr 2&4&9 \cr 5&@ \cr 7 \cr }
\syoung{1&3&6&8&@\cr 2&4&9&@ \cr 5&@ \cr 7 \cr }
\syoung{   1&3&6&8&10\cr 2&4&9&@&@ \cr 5&@ \cr 7 \cr }
\syoung{ 1&3&6&8&10&@&@&\dots\cr 2&4&9&@&@ \cr 5&@ \cr 7 \cr }
={T^@_*}$
\end{tabular}

\nopagebreak\vspace{.3cm}
\noindent
\begin{tabular}{rrl@{\qquad\qquad}rl}
12.&$T_1(\tw)=$& \young{ 1&3&5&6&9\cr 2&7&8 \cr
4&10  \cr }
 &$T_2(\tw)=$ &\young{ 1&3&6&8&10\cr 2&4&9 \cr 5
  \cr 7  \cr }
\end{tabular}
}\caption{An example of application of the mirabolic RSK algorithm}\label{fig:RSK}
\end{figure}

As result we have
{$\nu=\Sh(T_1(\tw))=(5,3,2);\quad  \nu'=\Sh(T_2(\tw))=(5,3,1,1);
\quad\theta=(\Sh(T^@_*))_-=(\infty,5,2,1)_-=(5,2,1). $}

\begin{thm}\label{fthm}
For any $\tw\in RB$ the triple $(\bt(\tw),T_1(\tw),T_2(\tw))$
obtained by the mirabolic algorithm described in~\ref{RSK2}
coincides with the triple $(\bt(\tw),T_1,T_2)$ defined in~\ref{mir}.
\end{thm}

\begin{proof}

Consider colored permutation $\tw_+\in RB_{3N}$ defined by the formulas
\begin{tabbing}
\hspace{4cm}\=$\tw_+\quad=$ \=  $ (w_+,\b_+)$ \\
\>$w_+(i)=$ \> $        \begin{cases}
                              i+2N      &  \text{ if $\quad i\le N$}\\
                              w(i-N)+N  &\text{ if $\quad N<i\le 2N$}\\
                              i-2N      & \text{ if $\quad i> 2N$}\\
                        \end{cases}$\\
\>$\b_+\quad=$ \> $\{i+N|i\in\b\}$
\end{tabbing}

Consider a general point $x\in Y_{\tw_+}, \quad x=(F_1,F_2,u_1,u_2,v,v^*)$.
Denote by $S$ the annihilator of $\sk[u_1^*]\cdot v^*$.
We are going to describe
the relative position of flags  $F_1\cap S$ and $F_2\cap S$.

\subsection{The relative position of flags  $F_1\cap S$ and $F_2\cap S$}
Define 2 sequences of subsets $\{\gamma_m\}$ and $\{\delta_m\}$ $\quad
(m\ge 1)$  inductively as follows:

1. $\gamma_1=\{1,\ldots,3N \}\setminus\b_+$.

2. $\delta_m$ consists of all $i\in\gamma_m$ such that there exists
no $j\in\gamma_m$ satisfying both inequalities
$j<i$ and $w_+(j)<w_+(i)$.

3. $\gamma_{m+1}=\gamma_m\setminus\delta_m$.

It is easy to check that $\delta_m\ne\varnothing$ iff $1\le m \le N$,
moreover, the minimal element of $\delta_m$ is equal to $m$ and the
maximal one is equal to $m+2N$. Define a permutation $w'_1:\{N+1,\ldots,3N\}
\to\{N+1,\ldots,3N\}$ as follows:

$w'_1(i)=\begin{cases}
         w_+(i)&\text{if $i\in\b_+$}\\
         w_+(j)\text{,   where $j=\max\{l\in\delta_m| l<i\}$}&
\text{if $i\in\delta_m$}
\end{cases}
$

\begin{lem}\label{relf}
The flags $F_1\cap S$ and $F_2\cap S$ are in relative position $w_1$.
\end{lem}

\begin {proof}
Choose a basis $e_1,\ldots,e_{3N}$ of $V_+$ such that $F_{1,i}=\langle e_1,
\ldots,e_i\rangle ;\quad F_{2,j}=\langle e_{w_+^{-1} (1)},\ldots,e_{w_+^{-1} (j)}\rangle$.
Denote by $\{e_i^*\}$ the dual basis. Then by sufficiently general choice of
the point $x$ and the basis $\{e_i\}$
we will have $(u_*)^mv^*=\sum\limits_{i\in\gamma_{m-1}}a_{m,i}e_i^*$
where the coefficients $a_{m,i}\ne 0$. Note that the space $S$ is the
intersection of kernels of functionals $(u^*)^mv^*$, where $0\le m\le N-1$.
Hence it is transversal to the spaces
$F_{1,N}$ and $F_{2,N}$. Therefore $i$-dimensional subspaces of the flags
$F_1\cap S$ and $F_2\cap S$ have the form  $F_{1,i+N}\cap S$ and
$F_{2,i+N}\cap S$.

\bigskip
Denote by $r_{i,j}(w'_1)$ the number of all $i'$ such that $i'\le i$ and
$w'_1(i')\le j$. Then to prove the lemma we have to show that
$\dim F_{1,i}\cap F_{2,j}\cap S=r_{i,j}(w'_1)$ for any
$i,j\in\{N+1,\ldots,3N\}$. Define $r_{i,j}(w_+)$ in the same way. Then
$\dim F_{1,i}\cap F_{2,j}=r_{i,j}(w_+)$.
Denote by $R_{i,j}$ the set of all $i'\le i$ such that $w_+(i')\le j$.
Then $ F_{1,i}\cap F_{2,j}$ has a basis $\{e_{i'}\}$ where $i'\in R_{i,j}$.

Note that if $m\le m'$ and $\delta_m\cap R_{i,j}\ne\varnothing$ then
$\delta_{m'} \cap R_{i,j}\ne\varnothing$, so we can find $k_{i,j}\ge 0$
such that $\delta_m \cap R_{i,j}\ne\varnothing$ iff $m\le k_{i,j}$.
Then $(u^*)^mv^*|_{F_{1,i}\cap F_{2,j}}\ne 0$ iff $m\le k_{i,j}$.
Moreover, for ${m=1,\ldots,k_{i,j}}$ these functionals are linearly
independent. By this reason  the space $F_{1,i}\cap F_{2,j}\cap S$
being the intersection of kernels of these functionals has dimension
$\dim F_{1,i}\bigcap F_{2,j}\bigcap S=\dim F_{1,i}\bigcap F_{2,j}-k_{i,j}=
r_{i,j}(w_+)-k_{i,j}$.

\bigskip
It remains to prove that $r_{i,j}(w_+)-k_{i,j}=r_{i,j}(w'_1)$.
We have the following equalities:

\smallskip
{$ R_{i,j}(w_+)=(R_{i,j}(w_+)\cap\b_+)\bigcup (\bigcup\limits_{m=1}^N
R_{i,j}(w_+)\cap\delta_m)$};

{$ R_{i,j}(w'_1)=(R_{i,j}(w'_1)\cap\b_+)\bigcup (\bigcup\limits_{m=1}^N
R_{i,j}(w'_1)\cap\delta_m)$}.

From the definition of $w'_1$ we obtain $R_{i,j}(w'_1)\cap\b=
R_{i,j}(w_+)\cap\b$. Besides,

in the case $m>k_{i,j}$ we have $R_{i,j}(w'_1)\cap\delta_m=
R_{i,j}(w_+)\cap\delta_m=\varnothing\quad;$

in the case $m\le k_{i,j}$ we have $R_{i,j}(w'_1)\cap\delta_m=
R_{i,j}(w_+)\cap\delta_m\setminus\{i_m \}$, where
$i_m$ is the minimal element of $R_{i,j}(w_+)\cap\delta_m$.

This implies that the set  $R_{i,j}(w'_1)$ can be obtained from the set
$R_{i,j}(w_+)$ by removing the
elements $i_1,\ldots,i_{k_{i,j}}$, whence we get the required equality:
$r_{i,j}(w'_1)=r_{i,j}(w_+)-k_{i,j}$.
\end {proof}

\subsection{}
Let, as before, $x=(F_1,F_2,u_1,u_2,v,v^*)$ be a general point of~$Y_{\tw_+}$, 
and put $ {S=(\sk[u_1^*]v^*)^\perp}$.
Let $u=u_1|_S=-u_2|_S\in\End S$, and let $T'_1$ and $T'_2$ be the standard tableaux
such that $F_1\cap S\in\Fl_{u,T'_1}$;  $F_2\cap S\in\Fl_{u,T'_2}$.

\begin{lem}
One can make  the flags $F_1\cap S$ and  $F_2\cap S$ to be any points of
varieties $\Fl_{u,T'_1}$ and $\Fl_{u,T'_2}$ by an appropriate choice
of a point $x$.
\end{lem}

\begin{proof}
Consider any $F_1^S\in\Fl_{u,T'_1}$ and $F_2^S\in\Fl_{u,T'_2}$ and let the
flags $F'_1, F'_2$ be defined~(for~$k=~1,~2)$  as follows:
$\begin{cases}
F'_{k,i}=F_{k,i}&             \text{if $i\le N$}\\
F'_{k,i}=F^S_{k,i-N}+F_{k,N}& \text{if $i>N$.}
\end{cases}$

Then $x'=(F'_1,F'_2,u_1,u_2,v,v^*)\in Y$. Note that the correspondence
${(F'_1,F'_2)\mapsto x'}$ defines a map
$f:\Fl_{u,T'_1}\times\Fl_{u,T'_2}\to Y$.

Since $\Z{u,T'_1}\times\Z{u,T'_2}$ is irreducible, the image of $f$
belongs to one irreducible component of $Y$.
As this image contains the point $x$, it lies in $Y_{\tw_+}$. Finally,
as $F_1^S$ and $F_2^S$ are arbitrary points of
$\Z{u,T'_1}$ and $\Z{u,T'_2}$, replacing $x$ by $x'$ proves the lemma.
\end{proof}

\subsection{}
According to Lemma~\ref{relf}, the relative position of
$F_1\cap S$ and $F_2\cap S$
is given by the permutation $w_1$, so using the result of
Spaltenstein~(\cite{S}), we see that the pair $T'_1,T'_2$ corresponds
to $w_1$ by the classical RSK correspondence.

Now note that the spaces $\1{2N}$ and $\2{2N}$ are invariant
with respect to both operators $u_1$ and $u_2$.
Let $V=\1{2N}\cap\2{2N}$. Then a sixtuple
$(F_1\cap V,F_2\cap V, u_1|_V,u_2|_V,v,v^*|_V)$ is a general point
of variety $Y_\tw$. Let $(\bt(\tw),T_1(\tw),T_2(\tw))$ be the triple
defined in~\ref{mir} and
$\bt(\tw)=(\nu,\theta,\nu')$. Then $F_1\cap V\in\Z{u_1|_{V},T_1(\tw)}$ , $F_2\cap V\in\Z{u_2|_{V},T_1(\tw)}$
and $\theta$ is the type of the nilpotent $u_1|_{V/\sk[u_1]v}$.

\begin{lem}\label{star}
The tableaux $T_1(\tw)$ and  $T_2(\tw)$ are obtained of the tableaux
$T'_1$ and $T'_2$ by removing the numbers
$N+1,\ldots,2N$.	

\end{lem}

\begin{proof}
By symmetry, it suffices to prove the lemma for
$T_1(\tw)$ which we will denote by $T$ for short.
Recall that $T'_1$ is defined by the condition
$F_1\cap S\subset\Z{u_1|_S,T'_1}$.
So denoting by $\tT_1$ the tableau obtained from $T'_1$ by removing
the numbers greater than $N$, we have
$$F_1\cap S\cap \1{2N}\in \Z{u_1|_{S\cap \1{2N}},\tT_1}.$$
Note that the spaces $V$ and $S\cap \1{2N}$ are both complementary to
$\1{N}$ inside $\1{2N}$. Therefore they
can be identified with $\1{2N}/\1{N}$. Under this identification the
operators $u_1|_V$ and $u_1|_{S\cap\1{2N}}$ go
to the same operator $u_1|_{\1{2N}/\1{N}}$. Similarly,
the flags $F_1\cap V$ and $F_1\cap S\cap\1{2N}$
go to the same flag $(F_1\cap\1{2N})/\1{N}$.
From this we obtain that
$$(F_1\cap\1{2N})/\1{N}\in\Z{u_1|_{\1{2N}/\1{N}},T_1}\qquad
\textrm{and}\qquad (F_1\cap\1{2N})/\1{N}\in\Z{u_1|_{\1{2N}/\1{N}},\tT_1}. $$
Now $F_1$ being a general point of a certain component
of the variety $\Z{u_1}$ we obtain that $T_1=\tT_1$.
\end{proof}

\begin{lem}\label{shapes}
$$\theta=(\Sh(T'_1))_-=(\Sh(T'_2))_-.$$
\end{lem}

\begin{proof}
By definition  of the tableaux $T'_1$ and $T'_2$, their shape coincides with
the type of the nilpotent $u=u_1|_S=u_2|_S$.
On the other hand, $\theta$ is the type of $u_1|_{V/\sk[u_1]v}$.
Define $L:=\sk[u]v$ and consider  the space
$D=(\1{N}+\2{N})\cap S+L$. It is invariant under $u$, moreover $u|_D$
has only one Jordan block (it can be checked directly).
Besides, $D\cap V=L$ and $D+V=S$. Therefore $u|_{V/L}$ has the same type as
$u|_{S/D}$. Define $d:=\dim D$. Then from the equalities
$S=D+V,\quad\dim V=N,\quad \dim S=2N$, it follows $d\ge N$.
So $d$ is the least power of $u$ vanishing on $S$. Hence, the type
of $u|_{S/D}$ is obtained from the type of $u|_S$ by removing the
maximal part of the partition.
\end{proof}

\subsection{The completion of proof of Theorem~\ref{fthm}}
Let $(v^c,\theta^c,(\nu')^c,T_1^c,T_2^c)$ be the result of application
to $\tw$ of the algorithm described in~\ref{RSK2}.
We have to prove that this quintuple coincides with
$(v(\tw),\theta(\tw),\nu'(\tw),T_1(\tw),T_2(\tw))$.

Note that the result of application of algorithm~\ref{RSK2} will not change
if instead of infinite row of symbols ``@''
we will take finite sequence $N+1,\ldots,2N$. Then $i+N\in\delta_m$
iff at the $i$-th step of the algorithm $w(i)$ is
being inserted into the $m$-th position of $r^@$. In this case $w_1(i)$
is the number inserted into $T^@$ at the $i$-th step of the algorithm.
Hence, if we apply to $w_1$ the classical RSK algorithm and after that
throw out from the tableaux $T_1(w_1)$  and
$T_2(w_1)$ all the numbers greater than $N$ then we obtain the same pair
of tableaux as the pair $T_1^c$ and $T_2^c$ obtained by the algorithm
~\ref{RSK2}.

Moreover, the partition $\theta^c$ has the form
$\theta^c=(\Sh(T_1(w_1)))_-$. We have proved above that
$$T_1(w_1)=T'_1;\quad T_2(w_1)=T'_2;\quad \theta=(\Sh(T'_1))_-.$$
In view of Lemma~\ref{star} we obtain
$$T_1^c=T_1(\tw);\quad T_2^c=T'_2(\tw)\quad\text{and}\quad
\theta^c=\theta(\tw).$$

The proof of Theorem~\ref{fthm} is completed.
\end{proof}

\section{Hecke algebra and mirabolic bimodule}

\subsection{}
\label{4.1}
Let $\S$ be a finite set and $E$ be a vector space over
$\CC$ with basis $\{e_\a\}_{\a\in\S}$.
Then the algebra $\End(E)$ of all linear operators on $E$ can be
described
as the algebra of $\CC$-valued functions on $\S\times\S$ with the
multiplication given by convolution: $$(f*g)(\a,\b)=
\sum_{\g\in\S}f(\a,\g)g(\g,\b)$$
 If a finite group $G$ acts on $\S$ then it also acts on $E$ and $\End(E)$.
Denote by $H= \End_G(E)\subset\End(E)$ the algebra of $G$-invariants in
$\End(E)$. It consists of all functions on $\S\times\S$ that are constant
on each $G$-orbit.

Now let $\sk=\FF_q$ be a finite field of $q$ elements. Let $V,X,Y$ be
as in the previous section. Let $\S$ be the set of $\sk$-points of
$\Fl(V)$ and $G=GL(V)$. Then the algebra $H$ from the previous
paragraph is called Hecke algebra. It has a basis consisting of
characteristic functions of orbits. Denote by $T_w$ the characteristic
function of $\O_w$ considered as an element of $H$. Now consider the
vector space $R$ of $G$-invariant $\CC$-valued functions on $X(\sk)$
where $X=\Fl(V)\times\Fl(V)\times V$. It has a natural structure of
$H$-bimodule. Namely, if $f\in H,g\in R$ then
$$(f*g)(F_1,F_2,v)=\sum_{F\in[\Fl(V)](\sk)}f(F_1,F)g(F,F_2,v),$$
$$(g*f)(F_1,F_2,v)=\sum_{F\in[\Fl(V)](\sk)}g(F_1,F,v)f(F,F_2).$$

If $\tw\in RB$, let $T_{\tw}\in R$ denote the characteristic function
of the corresponding orbit $\O_\tw\subset X$.
Note that the involutions $(F_1,F_2)\invl (F_2,F_1)$ and
$(F_1,F_2,v)\invl (F_2,F_1,v)$ induce anti-automorphisms of the
algebra $H$ and the bimodule $R$. These anti-automorphisms send
$T_w$ to $T_{w^{-1}}$ and $T_\tw$ to $T_{\tw^{-1}}$ where
$\tw^{-1}=(w^{-1},w(\b))$ for $\tw=(w,\b)$.

\subsection{Explicit formulas for the action of $\bH$ in $\bR$}
We are now going to compute the $H$-action on~$R$ in the basis
$\{T_\tw\}$. It is known that the algebra~$H$ is generated by the
elements $T_{s_i}$ where $s_i=(i,i+1)$ is the elementary transposition.
So, it suffices to compute $T_{s_i} T_\tw$ and $T_\tw T_{s_i}$. We will
compute only $T_\tw T_{s_i}$, since the other product can be obtained by
applying the above anti-automorphisms.

\begin{prop}[cf.~\cite{Sol}, Th.~6.6]\label{thm_expl}
Let $\tw=(w,\b)\in RB$ and let $s=s_i\in\gS_N$, $i\in\{1,\dots,N-1\}$.

Denote $\tw s = (ws, s(\b))$ and $\tw'=(w,\b\btu\{i+1\})$.\footnote
{Here $\btu$ is the {\em symmetric difference}:   $A\btu B  := (A\setm B) \cup (B\setm A)
=(A\cup B)\setm(A\cap B)$.}
 Let $\s=\s(\tw)$
and $\s'=\s(\tw s)$ be given by~\eqref{sigma}. Then
{\small\rm \eql{eq_expl}{  T_\tw T_s =
\begin{cases}
T_{\tw s} &\text{if $ws > w$ and  $i+1 \not\in \s'$,}  \\
T_{\tw s} + T_{(\tw s)'} &
             \text{if $ws > w$ and  $i+1 \in \s'$,}  \\
T_{\tw'} + T_{\tw's} &
             \text{if $ws < w$ and  $\b \cap \iota = \{i\}$,}  \\
(q-1) T_\tw + q T_{\tw s} &
             \text{if $ws < w$ and  ($i \not\in \s$ or $i+1\in\b\setminus\s$),}  \\
(q-2) T_\tw + (q-1) (T_{\tw'} + T_{\tw s}) &
             \text{if $ws < w$ and  $\iota \subset \s$}  \\
\end{cases} }}
where $\iota = \{i,i+1\}$.
(Note that the condition in the fourth case is equivalent to ``$ws < w$, $\tw s\in RB$ and $\s'=s(\s)$''.)
\end{prop}

\subsection{Tate sheaves}
It is well-known that $H$ is the specialization under $\bq\mapsto q$
of a $\BZ[\bq,\bq^{-1}]$-albebra $\bH$.
The formulas~(\ref{eq_expl}) being polynomial in $q$, we may (and will) view
$R$ as the specialization under $\bq\mapsto q$ of a
$\BZ[\bq,\bq^{-1}]$-bimodule $\bR$ over the $\BZ[\bq,\bq^{-1}]$-algebra $\bH$.
We consider a new variable $\bv,\ \bv^2=\bq$, and extend the scalars to
$\BZ[\bv,\bv^{-1}]:\ \CH:=\BZ[\bv,\bv^{-1}]\otimes_{\BZ[\bq,\bq^{-1}]}\bH;\
\CR:=\BZ[\bv,\bv^{-1}]\otimes_{\BZ[\bq,\bq^{-1}]}\bR$.

Recall the basis $\{H_w:=(-\bv)^{-\ell(w)}T_w\}$ of $\CH$ (see e.g.~\cite{So}),
and the Kazhdan-Lusztig basis $\{\utH_w\}$ ({\em loc. cit.});
in particular, for
a simple transposition $s,\ \utH_s=H_s-\bv^{-1}$.
For $\tw\in RB$, we denote by $\ell(\tw)$ the difference $\dim(\Omega_\tw)-
\bn$, where $\bn:=\frac{N(N-1)}{2}=\dim(\Fl(V))$.
We introduce a new basis $\{H_\tw:=(-\bv)^{-\ell(\tw)}T_\tw\}$ of $\CR$.
In this basis the right action of the Hecke algebra generators $\utH_s$
takes the form:

\begin{prop}
\label{thm_explicit}
Let $\tw=(w,\b)\in RB$ and let $s=s_i\in\gS_N$, $i\in\{1,\dots,N-1\}$.
Denote $\tw s = (ws, s(\b))$ and $\tw'=(w,\b\btu\{i+1\})$. Let $\s=\s(\tw)$
and $\s'=\s(\tw s)$ be given by~\eqref{sigma}. Then
{\small\rm \eql{eq_explicit}{  H_\tw\utH_s =
\begin{cases}
H_{\tw s}-\bv^{-1}H_\tw &\text{if $ws > w$ and  $i+1 \not\in \s'$,}  \\
H_{\tw s} - \bv^{-1}H_{(\tw s)'}-\bv^{-1}H_\tw &
             \text{if $ws > w$ and  $i+1 \in \s'$,}  \\
H_{\tw'}-\bv^{-1}H_\tw-\bv^{-1}H_{\tw's} &
             \text{if $ws < w$ and  $\b \cap \iota = \{i\}$,}  \\
H_{\tw s}-\bv H_\tw &
             \text{if $ws < w$ and  ($i \not\in \s$ or $i+1\in\b\setminus\s$),}  \\
(\bv^{-1}-\bv)H_\tw + (1-\bv^{-2})(H_{\tw'} + H_{\tw s}) &
             \text{if $ws < w$ and  $\iota \subset \s$}  \\
\end{cases} }}
where $\iota = \{i,i+1\}$.
\end{prop}

It is well known that $\CH$ is the Grothendieck ring (with respect to
convolution) of the derived constructible $G$-equivariant category of
Tate Weil $\ol{\mathbb Q}_l$-sheaves on $\Fl(V)\times\Fl(V)$,
and multiplication by $\bv$
corresponds to the twist by $\ol{\mathbb Q}_l(-\frac{1}{2})$
(so that $\bv$ has
weight 1), see e.g.~\cite{BGS}. In particular, $H_w$ is the class of the
shriek extension of $\Ql[\ell(w)+\bn](\frac{\ell(w)+\bn}{2})$
from the corresponding
orbit, and $\utH_w$ is the selfdual class of the Goresky-MacPherson extension
of $\Ql[\ell(w)+\bn](\frac{\ell(x)+\bn}{2})$ from this
orbit. Similarly, we will prove that $\CR$ is the Grothendieck group
of the derived constructible $G$-equivariant category of Tate Weil
$\ol{\mathbb Q}_l$-sheaves on $X$, and $\CH$-bimodule structure is given by
convolution. In particular, $H_\tw$ is the class of the
star extension of
$\Ql[\ell(\tw)+\bn](\frac{\ell(\tw)+\bn}{2})$
from the orbit $\Omega_\tw\subset X$. We will denote by
$j_{!*}\Ql[\ell(\tw)+\bn](\frac{\ell(\tw)+\bn}{2})$ the selfdual
Goresky-MacPherson extension of $\Ql[\ell(\tw)+\bn](\frac{\ell(\tw)+\bn}{2})$
from $\Omega_\tw\subset X$, and we will denote by $\utH_\tw$ its class
in the Grothendieck group.

Recall that a $G$-equivariant constructible Weil complex $F$ on $X$ is
called {\em Tate} if any cohomology sheaf of its restriction $i_\tw^*F$ and
corestriction $i_\tw^!F$
to any orbit $\Omega_\tw$ admits a filtration with successive quotients
of the form
$\Ql(m),\ m\in\frac{1}{2}\BZ$. If for any $\tw\in RB$ the sheaf
$j_{!*}\Ql[\ell(\tw)+\bn](\frac{\ell(\tw)+\bn}{2})$ is Tate, then the
shriek extension $j_!\Ql[\ell(\tw)+\bn](\frac{\ell(\tw)+\bn}{2})$ is Tate
as well (see Remark between Lemmas 4.4.5 and 4.4.6 of~\cite{BGS}).
Note also that the $G$-equivariant geometric fundamental group of any orbit
$\Omega_\tw$ is trivial.
Hence the classes $H_\tw=[j_!\Ql[\ell(\tw)+\bn](\frac{\ell(\tw)+\bn}{2})]$
do form a $\BZ[\bv,\bv^{-1}]$-basis of the Grothendieck group of
$G$-equivariant
Tate sheaves on $X$, and this Grothendieck group is isomorphic to $\CR$.

In order to prove the Tate property of
$j_{!*}\Ql[\ell(\tw)+\bn](\frac{\ell(\tw)+\bn}{2})$, we need to study
certain analogues of Demazure resolutions of the orbit closures
$\overline\Omega_\tw$.

\subsection{Demazure type resolutions}
\label{ny}
We consider the elements $\tw_i=(w,\beta_i)\in RB$ such that
$w=\id$ (the identity permutation), and $\beta_i=\{1,\ldots,i\}$, where
$i=0,\ldots,N$. We set
$\utH_{\tw_i}:=\sum_{0\leq j\leq i}(-\bv)^{j-i}H_{\tw_j}$. This is the class
of the selfdual (geometrically constant) IC sheaf on the closure of the
orbit $\Omega_{\tw_i}$.

We fix $k \quad(0\le k\le N)$, and
a pair of sequences  $i_1,\ldots,i_r$ and  $j_1,\ldots,j_s$
of integers between 1 and $N-1$.
Let  $S=S_{i_1,\ldots,i_r}^{j_1,\ldots,j_s;k}$ be a variety of collections
of flags and vectors $(F_0,\ldots,F_r,F_0',\ldots,F_s',v)$ such that:

1.\quad$(F_r,F_0',v) \in \zowt {k}$;

2.\quad$(F_{p-1},F_p)\in\zos{i_p}$ for any $p\in{\{1,\ldots,r\}}$;

3.\quad$(F'_{q-1},F'_q)\in\zos{j_q}$ for any $q\in{\{1,\ldots,s\}}$.

In other words,
$$S=\zos{i_1} \uf \ldots\uf\zos{i_r}\uf\zowt{k}\uf\zos{j_1}\uf
\ldots\uf\zos{j_s}$$

Consider a map $\phi=\phi_{i_1,\ldots,i_r}^{j_1,\ldots,j_s} :
S_{i_1,\ldots,i_r}^{j_1,\ldots,j_s;k} \to X$ which takes
 $F_0,\ldots,F_r,F_0',\ldots,F_s'$ to $(F_0,F'_s,v)$.

\begin{prop}
\label{ABC}

For any $\tw\in RB$ there exist $i_1,\ldots,i_r;j_1,\ldots,j_s$
and $k$ such that:

a) $\phi ( S)=\zowt {} $, moreover
$\phi$ is an isomorphism over $ \owt{} $.

b) The sheaf$\quad\phi_*(\Ql)$ is Tate.
\end {prop}

\begin{proof}

a) We proceed by induction in $\ell(\tw)=\dim\owt{}-\bn$.
Assume the proposition~\ref{ABC} is true for any
$\tw '$ such that $\ell(\tw') < \ell(\tw)$.

Let $\tw=(w,\b)$. If $w=\id$ then $\tw=\tw_k$ for some k. We choose
$i_1,\ldots,i_r$ and $j_1,\ldots,j_s$ to be the empty sequences.
Then the map $\phi$ is
an embedding $\owt{}\hookrightarrow X$, and the proposition is true.
Otherwise $(w\ne \id)$ it is easy to show that either $ws_i<w$ and
$\tw s_i=(ws_i,s_i(\b))\in RB$, or $s_iw<w$ and $(s_iw,\b)\in RB$.
Without loss of generality we can restrict ourselves to the first case
(the second one is obtained replacing
$\tw$ by $\tw^{-1}=(w^{-1},w(\b))$ ).

Let $S'= S_{i,i_1,\ldots,i_r}^{j_1,\ldots,j_s;k},
\phi'=\phi_{i,i_1,\ldots,i_r}^{j_1,\ldots,j_s}  $. Then
\eql{z1} {S'=\zos{i}\uf S.}

By the induction hypothesis, $S$ contains an open dense set
mapping isomorphically by $\phi$ onto $\owts{i}$.
It follows that a map $\O_s\uf\owts{i}\to X$
has an image lying in $\zowt{}$, moreover,
this map is an isomorphism over $\owt{}$. According to ~(\ref{z1}),
$S'$ contains an open
dense subset isomorphic to $\os{i} \uf\owts{i}$,
hence $\phi'(S')=\zowt{}$,
and $\phi^{-1}(\owt{})$ is isomorphic to $\owt{}$.

\medskip
b) We will prove by induction that any fiber of $\phi$
is paved by the pieces isomorphic to
$\AAbb^k\times\GG^n_m$. Moreover, the union of pieces constructed at each
step is a closed subvariety of the fiber.
Since the cohomology of $\AAbb^k\times\GG^n_m$ is Tate, the assertion will
then follow.

If $r=s=0$ then any nonempty fiber is just a point,
and the statement is obvious.
Otherwise without loss of generality we can assume $r>0$.

Let
$S=S_{i_1,i_2,\ldots,i_r}^{j_1,\ldots,j_s;k}$ and
$S'=S_{i_2,\ldots,i_r}^{j_1,\ldots,j_s;k}$.
We have a commutative diagram:

\begin{equation}
\label{d1}
\xymatrix{
S=\O_{s_{i_1}}\times_{{}_ {\Fl(V)}} S'
\ar[rr]^(.43){\tilde {\pi}}
\ar[rd]_\phi
  && \O_{s_{i_1}}\times \Fl(V)\times V
\ar[dl]|{\hspace{.4cm}\psi=pr_1\times id_{\Fl(V)}\times id_V}
\\
&X\\
}
\end{equation}

where $\tilde{\pi}(F_0,\ldots,F_r,F'_1,\ldots,F'_s)=((F_0,F_1),F'_s,v)$.

\bigskip
It is easy to see that the fibers of the map $\psi$ are isomorphic to $\PP^1$.
For each point $x\in X$ we obtain the corresponding map
$\pi:\phi^{-1}(x)\to\psi^{-1}(x)\cong\PP^1$.
We have the following commutative diagram, whose
middle part coincides with diagram~\eqref{d1}:
\begin{equation}
\label{d2}
\xymatrix{
\phi^{-1}(x) \ar[rrrr]^\pi  \ar[rrdd] \ar@{^{(}->}@<-1ex>[ddd]
&&&& \psi^{-1}(x)\cong\PP^1 \ar[ddll]
        \ar@{^{(}->}@<2ex>[ddd]
\\
\\
&&{\strut x} \ar@{^{(}->}[ddd]\\
S=\O_{s_{i_1}}\times_{{}_ {\Fl(V)}} S'
\ar[rrrr]^(.43){\tilde {\pi}}
\ar[rrdd]_\phi
\ar@<-1ex>[ddd]_{projection}
  &&&& \O_{s_{i_1}}\times \Fl(V)\times V
\ar@<2ex>[ddd]^{pr_2\times id_{\Fl(V)}\times id_V}
\ar[ddll]|{\hspace{.4cm}\psi=pr_1\times id_{\Fl(V)}\times id_V}
\\
\\
&&X\\
S' \hspace{.5cm} \ar[rrrr]_{\phi'} &&&&{\hspace {.6cm} X}\\
}
\end{equation}
All the 4 squares in this diagram are Cartesian.

Denote by $\varkappa : \psi^{-1}(x)\to X$ the composition of maps
from the commutative diagram~\ref{d2}. This map is an embedding.

For each point $y\in\psi^{-1}(x)$ we have
$\pi^{-1}(y)\cong(\phi')^{-1}(\varkappa(y))$.
By the induction hypothesis, all the fibers of $\phi'$
can be decomposed into pieces of required form.
If $x=(F,F',v)$ then the image $L$ of the map $\varkappa$ consists of triples
$(F,F'',v)$ such that
$(F',F)\subset\zos{i}$. For $x'\in X$ the fiber
$(\phi')^{-1}(x')$ depends only on the orbit $\ot{u}$ which contains $x'$.
The line $L$ can intersect 2 or 3 such orbits; one intersection is open
in $L$, and any other intersection is just one point.

Let $U=L\bigcap \ot{u}$ be the intersection open in $L$.
Since all the fibers of $\pi$ admit a required decomposition, it is enough
to construct the decomposition of the set $\pi^{-1}(U')\cong\phi{-1}(U)$
where $U'=\varkappa{-1}(U)\cong U$. This follows from the fact that the bundle
$(\phi')^{-1}(U)\to U$ is trivial.

Indeed, in this case for any $x'\in U$ we have
$\phi'^{-1}(U)\cong U\times(\phi')^{-1}(x')$,
because $\phi'^{-1}(x')$ admits the required decomposition, and
$U$ is isomorphic to either $\AAbb^1$ or $\GG_m$.

It remains to prove the triviality of the bundle $(\phi')^{-1}(U)\to U$.
Note that the bundle $S'\to X$ is $GL(V)$-equivariant. Choose a point
$x'\in U$ and consider a map $GL(V)\to X$, given by $g\to g\cdot x'$.
Then the induced bundle $S\times_{{}_X} GL(V)\to GL(V)$ is trivial, so it is
enough to prove that there exists the dotted arrow in the diagram

\begin{equation}\label{d3}
\xymatrix{
U\ar@{^{(}->}[rr] \ar@{-->}[rd] &&X\\
&GL(V)\ar[ur]_{\cdot x'}
}
\end{equation}

This can be checked directly.

So the proof of proposition~\ref{ABC} is finished.

\end{proof}

\begin{cor}\label{corABC}
Bimodule $\R$ is generated by the elements $e_i=T_{\tw_i}$.
\end{cor}

\noindent{\it Proof.\/} We prove by induction on $\ell(\tw)$ that
$T_{\tw}\in \sum_i \H e_i\H$.

Choose $i_1,\ldots,i_r,j_1\ldots,j_s$ and $k$ as in Proposition~\ref{ABC}.
Then
$$T_{j_s}\cdot\ldots\cdot T_{j_1}\cdot T_{\tw_k}\cdot
T_{i_r}\cdot\ldots\cdot T_{i_1}=T_{\tw}+
\sum_{{\tilde u}<\tw}a_{\tilde u}T_{\tilde u}$$
where $a_{\tilde u}\in\BZ[\bv,\bv^{-1}]$, and
${\tilde u}<\tw$ means $\ot{u}\subset{\overline {\ot{w}} }$.
The left hand side of this equality belongs to $\H e_k \H$ (recall
that $e_k=T_{\tw_k}$). Besides, by the induction hypothesis,
$T_{\tilde u}\in \sum_i \H e_i \H$
  for each ${\tilde u}<\tw$. Hence, $\sum_{\tilde u}
a_{\tilde u}T_{\tilde u}\in \sum_i \H e_i \H$.

From this we can conclude that  $T_{\tilde u}\in \sum_i \H e_i \H$.   \qed

\begin{cor}
\label{pop}
For any $\tw\in RB$, the sheaf
$j_{!*}\Ql[\ell(\tw)+\bn](\frac{\ell(\tw)+\bn}{2})$ is Tate.
\end{cor}

\proof
Follows from Proposition~\ref{ABC}.b) by the Decomposition Theorem.
\qed

\begin{cor}
\label{balda}
The Grothendieck group of the derived constructible $G$-equivariant category of
Tate Weil $\ol{\mathbb Q}_l$-sheaves on $X$ is isomorphic to $\CR$ as an
$\CH$-bimodule with respect to convolution.
\end{cor}

\subsection{Duality and the Kazhdan-Lusztig basis of $\CR$}
Recall the involution (denoted by $h\mapsto\overline{h}$) of $\CH$ which
takes $\bv$ to $\bv^{-1}$ and $\utH_w$ to $\utH_w$. It is induced by the
Grothendieck-Verdier duality on $\Fl(V)\times\Fl(V)$. We are going to
describe the involution on $\CR$ induced by the Grothendieck-Verdier
duality on $X$.

Recall the elements $\tw_i=(w,\beta_i)\in RB$ such that
$w=\id$ (the identity permutation), and $\beta_i=\{1,\ldots,i\}$, where
$i=0,\ldots,N$. We set
$\utH_{\tw_i}:=\sum_{0\leq j\leq i}(-\bv)^{j-i}H_{\tw_j}$. This is the class
of the selfdual (geometrically constant) IC sheaf on the closure of the
orbit $\Omega_{\tw_i}$.

\begin{prop}
\label{duality}
a) There exists a unique involution $r\mapsto\overline{r}$ on $\CR$ such
that $\overline\utH_{\tw_i}=\utH_{\tw_i}$ for any $i=0,\ldots,N$, and
$\overline{hr}=\overline{h}\overline{r}$, and
$\overline{rh}=\overline{r}\overline{h}$ for any $h\in\CH$ and $r\in\CR$.

b) The involution in a) is induced by the Grothendieck-Verdier duality on $X$.
\end{prop}

\proof
The uniqueness in a) follows since $\CR$ is generated as an
$\CH$-bimodule by the set
$\{\utH_{\tw_i},\ i=0,\ldots,N\}$, according to Corollary~\ref{corABC}.
Now the Grothendieck-Verdier duality on $X$ clearly induces the involution
on $\CR$ satisfying a); whence the existence and b). \qed

Now let $\tw_1<\tw_2$ stand for the adjacency Bruhat order on $RB$
described combinatorially in~\cite{M},~section~1.2.

\begin{prop}
\label{KL basis}
a) For each $\tw\in RB$ there exists a unique element $\utH_\tw\in\CR$
such that $\overline\utH_\tw=\utH_\tw$, and
$\utH_\tw\in H_\tw+\sum_{\tilde{y}<\tw}\bv^{-1}
\BZ[\bv^{-1}]H_{\tilde{y}}$.

b) For each $\tw\in RB$ the element $\utH_\tw$ is the class of the
selfdual $G$-equivariant IC-sheaf with support $\bar{\Omega}_\tw$.
In particular, for $\tw=\tw_i$, the element $\utH_{\tw_i}$ is consistent
with the notation introduced before Proposition~\ref{duality}.
\end{prop}

\proof
a) is a particular case of~\cite{Lu},~Lemma~24.2.1.

b) We already know that $H_\tw$ is the class of
$j_!\Ql[\ell(\tw)+\bn](\frac{\ell(\tw)+\bn}{2})$, and
$j_{!*}\Ql[\ell(\tw)+\bn](\frac{\ell(\tw)+\bn}{2})$ is Tate.
Now b) follows from the Beilinson-Bernstein-Deligne-Gabber purity theorem
by the standard argument (see e.g.~\cite{Be},~section~6). \qed

\subsection{Pointwise purity}
We are now going to show that the sheaves
$j_{!*}\Ql[\ell(\tw')+\bn](\frac{\ell(\tw')+\bn}{2})$
are pointwise pure.
We choose a point $F_0\in\Fl(V)$ and denote by $B\in\GL(V)$ the corresponding
Borel subgroup.
We consider the preimage of $F_0$ under the second projection
$X':=X_{F_0}=pr_2^{-1}(F_0)
\subset X$, $X_{F_0}=\Fl(V)\times \{F_0\}\times V\cong \Fl(V)\times V$.
 Fix $\tw\in RB$, and let $\O'=\O_{\tw,F_0}=\O_\tw\cap X_{F_0}$
be the corresponding
$B$-orbit in~$X_{F_0}$. Choose a point $x\in X_{F_0}$.

\begin{lem} \label{Gm-action}
There exists a one-parameter subgroup $\chi:\ \GG_m\to B$ preserving $x$,
whose eigenvalues in the normal space $N_{x,X'}\O'$ are all positive.
\end{lem}

\proof
Let $\tw=(w,\b)$, $x=(F_0,F_1,v)$.
Choose a basis $e_1,\dots,e_N$ of~$V$ such that
$F_{0,i}=\langle e_1,\dots,e_i\rangle$,
$F_{1,i}=\langle e_{w(1)},\dots,e_{w(i)}\rangle$,
$v=\sum_{i\in\s} e_i$ where $\s$ is defined by~\eqref{sigma}.
Define sets $\b_r,\g_r,\s_r,\d_i \subset\{1,\dots,N\}$,
 $i=1,2,\dots$ inductively as follows:
\begin{itemize}
\item $\b_1=\b$, $\g_1=\{1,\dots,N\}\setminus\b$;
\item $\s_r=\{i\in\b_r \mid \nexists j\in\b_r\:
(j>i\:\text{and}\:w(j)>w(i)) \}$;\\
   $\d_r=\{i\in\g_r \mid \nexists j\in\g_r\:(j<i\:\text{and}\:w(j)<w(i)) \}$;
\item $\b_{r+1}=\b_r\setminus\s_r$, $\g_{r+1}=\g_r\setminus\d_r$.
\end{itemize}
Let $(k_1,\dots,k_N)$ be the integer sequence given by $k_i=1-r$ if $i\in\s_r$,
and $k_i=r$ if $i\in\d_r$. Then it is not hard to verify that the homomorphism
$t\mapsto\operatorname{diag}(t^{k_1},\dots,t^{k_N})$ satisfies the required
condition.
\qed

\begin{prop}
\label{varshav}
The intersection cohomology sheaf
$j_{!*}\Ql[\ell(\tw')+\bn](\frac{\ell(\tw')+\bn}{2})$ of an orbit
$\Omega_{\tw'}\subset X$ is pointwise pure of weight zero.
\end{prop}

\proof
We denote the locally closed embedding $\Omega_{\tw',F_0}\hookrightarrow X'$
by $a$. The statement of the proposition is equivalent to pointwise purity
of $a_{!*}\Ql$. We choose a point $x\in\Omega'=\Omega_{tw,F_0}\subset
\overline{\Omega}_{\tw',F_0}$.
Due to $B$-equivariance, it suffices to prove the purity of the
stalk of $a_{!*}\Ql$ at $x$. We claim that there is a locally closed
subvariety $Y'\subset X'$ such that a) $Y'\cap\Omega'=x$; b) $Y'$ is smooth
at $x$; c) the tangent space $T_xX'=T_x\Omega'\oplus T_xY'$; d) $Y'$ is stable
under the action of the one-parametric subgroup $\chi(\GG_m)$; moreover,
$Y'$ is contracted to $x$ by this action. The existence of $Y'$ with required
properties follows from Sumihiro equivariant embedding theorem~\cite{Su}.

We denote the locally closed embedding
$\Omega_{\tw',F_0}\cap Y'\hookrightarrow Y'$ by $a'$. The intersection
cohomology sheaf $a'_{!*}\Ql$ on $Y'$ is $\chi(\GG_m)$-equivariant, and hence
its stalk at $x$ is pure, see e.g.~\cite{Sp}.
The morphism $b:\ B\times Y'\to X',\ (g,y)\mapsto g(y)$ is smooth at the
point $(e,x)$ (where $e$ is the neutral element of $B$) due to the condition
c) above. It follows that $b^*a_{!*}\Ql|_{(e,x)}=pr_2^*a'_{!*}\Ql|_{(e,x)}$
where $pr_2:\ B\times Y'\to Y'$ is the second projection. We conclude that
the stalk $a_{!*}\Ql|_x$ is pure. The proposition is proved.
\qed

\begin{cor}
\label{pp}
We define the Kazhdan-Lusztig polynomials $P_{\tw,\tilde{y}}$ by
$\utH_\tw=\sum_{\tilde{y}\leq\tw}P_{\tw,\tilde{y}}H_{\tilde{y}}$.
Then all the coefficients of $P_{\tw,\tilde{y}}$ are nonnegative integers.
The coefficient of $\bv^{-k}$ in $P_{\tw,\tilde{y}}$ vanishes if
$k\not\equiv\ell(\tw)-\ell(\tilde{y})\pmod{2}$. \qed
\end{cor}

\subsection{The structure of the $\CH$-module $\CR$}
It is known that the algebra $\alg{\H}$ is isomorphic to the
group algebra of symmetric group $\QQ(\bv)[\gS_N]$.
Hence, the isomorphism classes
of irreducible modules over $\alg{\H}$
are indexed by the set of partitions of $N$. We denote by $V_{\nu}$ the
irreducible module corresponding to a partition $\nu$.

\begin{prop}
\label{algH}
$\alg{\H}$-bimodule $\alg{\R}$ has the following decomposition into
irreducible bimodules:
\eql{f3}{\alg{\R}=\bo_{(\tilde\nu,\tilde\theta,\tilde\nu')\in\bT}
V_{\nu}^*\otimes_{\QQ(\bv)} V_{\nu'} }
where the sum is taken over all the triples of partitions
$\nu,\theta, \nu'$ such that $|\nu|=|\nu'|=N$ and for any $i\ge 1$ we have
$\nu_i\ge\theta_i\ge \nu_i-1$; $\nu'_i\ge\theta_i\ge \nu'_i-1$.
\end{prop}

\begin{proof}
Choose a finitely generated $\BZ[\bv,\bv^{-1}]$-algebra
$\mathcal A\subset\QQ(\bv)$
such that $\H\otimes_{\BZ[\bv,\bv^{-1}]}\mathcal A$ is isomorphic to a
direct sum of matrix algebras over~$\mathcal A$, so that $V_\nu$ is defined
over~$\mathcal A$. Then it suffices to prove that this isomorphism holds after
the specialization
${\cdot}\otimes_{\mathcal A}\CC$ which takes
$\bv\mapsto\sqrt{q}$ where $q$ is a prime power such that
$\mathcal A\not\ni (\bv^2-q)^{-1}$.
In this case the left hand side of formula~\eqref{f3} can be interpreted as
$$\End_{P(\sk)}(E)\bo \H\otimes_{\BZ[\bv,\bv^{-1}]}\CC$$
where $P(\sk)\subset\GL_N$ is the stabilizer of $v\ne0$,
and $E$ is the vector space introduced in subsection~\ref{4.1}.
According to~\cite{Z}, Theorem~13.5.a), the irreducible components
of $P(\sk)$-module
$E$  are indexed by partitions $\theta$  with $|\theta|\leq N$.
We denote by $W_{\theta}$ the irreducible representation of $P(\sk)$
indexed by $\theta$. Denote by $U_{\nu}$ the irreducible
unipotent representation of
$G(\sk)$ indexed by $\nu$. Then the restriction of $U_{\nu}$ to $P(\sk)$
is a direct sum of the representations $W_{\theta}$ (with multiplicity one)
for all $\theta$ such that $\nu_i\ge\theta_i\ge \nu_i-1$ for any $i$ and $\theta\ne\nu$.

As a representation of the group $G(\sk),\ E$ admits a
decomposition as follows:
$$ E=\bo_{\nu} U_\nu\otimes V_\nu$$
(Here $G(\sk)$ acts on $V_{\nu}$ trivially). Therefore, as a
representation of $P(\sk)$, $E$ can be written in the form
$$ E=\bo_{\substack{\nu_i\ge\theta_i\ge \nu_i-1\\ \nu\ne \theta }}
W_\theta\otimes V_\nu $$
It follows that
$$\End_{P(k)}(E)=\mathop{{\bo}'}\limits_{(\tilde\nu,\tilde\theta,\tilde\nu')\in\bT}
V^*_\nu\otimes V'_\nu $$
where the sum $\bo'$ is the same as in formula~\eqref{f3},
but the case $\nu=\theta=\nu'$ is excluded. Besides, we have
$\H\otimes_{\BZ[\bv,\bv^{-1}]}\CC = \bo_\nu V^*_\nu\otimes V_\nu$.
Adding these equalities, we obtain the required result.

\end{proof}

\section{Bimodule KL cells}

\subsection{}
Consider all possible subbimodules of bimodule $\R$
spanned by subsets of basis ${\utH_{\tw}}$.
We say that two coloured permutation $\tw$ and  $\tw '$ belong to the same
Kazhdan-Lusztig bimodule cell if for each  such subbimodule
$\M$ we have  ${\utH_{\tw}}\in\M \Longleftrightarrow
{\utH_{\tw}'}\in\M.$ If instead of subbimodules we consider
left or right submodules then we obtain the definition of
left and right Kazhdan-Lusztig cells.

For $N=3$, a big part (for $\beta$ nonempty) of $RB$ is depicted
in~\cite{M}~1.3 with the help of Latin alphabet. It is a union of 13
two-sided KL cells: $\{\operatorname{max}\},\ \{z,u\},\ \{y,o,p,h\},\
\{x,t\},\ \{v,w,m,n\},\ \{s,i,k,b\},\ \{r,g\},\ \{q,f\},\ \{l,c\},\\
\{j,d\},\ \{e\},\ \{a\},\ \{\operatorname{min}\}$.
We take this opportunity to add two order relations missing in {\em loc. cit.}:
$c\lessdot l,\ r\lessdot v$.

\begin{conj}
\label{micro vs KL}
The bimodule KL cells coincide with the two-sided microlocal cells.
\end{conj}

We are only able to prove an inclusion in one direction, see Theorem~\ref{L}
below. First we have to formulate and prove a few lemmas.

\subsection{}
\label{defn Phi_i}
Consider the projection $\pi'_i\colon Y\to \Fl^{(i)}(V)\times\Fl(V)\times
V\times\N\times\N\times V^*$ where $\Fl^{(i)}(V)$ is the variety of flags consisting of subspaces of
$V$ which have dimensions $0,\dots,i-1,i+2,\dots ,N$. This projection sends a point
$(F_1,F_2,v,u_1,u_2,v^*)$ to $(\tilde{F_1},F_2,v,u_1,u_2,v^*)$ where
$\tilde{F_1}$ is obtained from $F_1$ by deleting the subspaces of dimensions
$i,i+1$. Let $Y_i=\pi'_i(Y)$.

Besides,  for each $i\in\{1,\dots,N-1\}$ consider the set $\Phi_i\subset RB $
defined as follows:
$(w,\b)\in \Phi_i$ iff for a general $(F_1,F_2,v,u_1,u_2,v^*)\in Y_{w,\b}$ we have $u_1|_{F_{1,i+1}/F_{1,i-1}}=0$. Denote by $s_i=(i,i+1)$ the elementary transposition.

\begin{lem}
\label{2}
a) Let $(w,\b)\in RB,i\in\{1,\dots,N-1\}$. Then $(w,\b)\in\Phi_i$ iff
$w(i)>w(i+1)$ and $\b\cap\{i,i+1\}\ne\{i\}$.

b)  Let $(w,\beta)$ and $(w',\beta')$ be distinct colored permutations.
Then $\pi'_i(Y_{w,\beta})=\pi'_i(Y_{w',\beta'})$ iff $w(j)=w'(j)$ when
$j\in\{1,\dots,N\}\setminus\{i,i+1,i+2\},$\\
${ \beta\btu \beta' := (\b\setminus\b')\cup(\b'\setminus\b)
\subset\{i,i+1,i+2\}}$,
and one of the following conditions is satisfied up to
interchanging $(w,\beta)$ and $(w',\beta')$:

\begin{enumerate}
\item
 $w(i)<w(i+2)<w(i+1), \\
\beta\cap\{i,i+1,i+2\}\in \{\varnothing,\{i\},\{i,i+2\},\{i,i+1,i+2\}\},\\
w'=ws_i,\ \beta'=s_i(\beta);$

\item
$ w(i+2)<w(i)<w(i+1),\\ \beta\cap\{i,i+1,i+2\}\in
\{\varnothing,\{i+2\},\{i,i+2\},\{i,i+1,i+2\}\},\\
w'=ws_{i+1},\ \beta'=s_{i+1}(\beta);$

\item
$w(i+2)<w(i+1)<w(i), \
\beta\cap\{i,i+1,i+2\}=\{i,i+2\},\\
w'=ws_{i+1},\ \beta'=s_{i+1}(\beta);$

\item
$w(i+2)<w(i)<w(i+1),\
 \beta\cap\{i,i+1,i+2\}=\{i\},\\
w'=ws_i,\ \beta'=s_i(\beta);$

\item
$w(i+2)<w(i+1)<w(i),\
 \beta\cap\{i,i+1,i+2\}=\{i\},\\
w'=w,\ \beta'=\beta\cup\{i+1\}.$

\end{enumerate}
\end{lem}

\begin{proof}
a) Let $\{e_i\}$ be a basis of $V$ and $x=(F_1,F_2,v)\in\O_{w,\b}$ be
the element given by~\eqref{F1},~\eqref{F2},~\eqref{v}. We must find the
conormal space $N^*_x \O_{w,\b}\subset T^*_x X$. It is isomorphic to the
space of triples $(u_1,u_2,v^*)$ such that $u_k$ is a nilpotent preserving
$F_k$ and $u_1+u_2+v\otimes v^*=0$. Let
\eql{ukv*}{u_k=\sum_{i,j=1}^N (u_k)_{i,j}E_{i,j},\quad v^*=\sum_{i=1}^N c_i
e_i^*} where $E_{i,j}$ is defined by~\eqref{Eij} and $e_i^*$ is the basis
dual to $e_i$. Then the last relation is equivalent to the fact that the
following conditions are satisfied:
$$\begin{cases}
c_i=0&\text{for $i\in\b$;}\\
(u_1)_{i,j}=-(u_2)_{i,j}&\text{for $i,j\in\b$ or $i,j\not\in\b$;}\\
(u_k)_{i,j}=0 &\text{for $i\not\in\b,\ j\in\b$;}\\
(u_1)_{i,j}+(u_2)_{i,j}=-c_j &\text{for $i\in\b,\ j\not\in\b,\ i<j,\
w(i)<w(j)$;}\\
(u_1)_{i,j}=-c_j,\ (u_2)_{i,j}=0 &\text{for $ i\in\b,\ j\not\in\b,\
w(i)>w(j)$;}\\
(u_1)_{i,j}=0,\ (u_2)_{i,j}=-c_j &\text{for $ i\in\b,\ j\not\in\b,\ i>j$}.
\end{cases}$$
If we substitute $j=i+1$, we obtain the statement a) of the lemma.

\bigskip
b) Note that the fiber $(\pi'_i)^{-1}(\tilde y)$ over an arbitrary point $\tilde y=(\tilde F_1,F_2,v,\linebreak u_1,u_2,v^*)\in Y_i$ is isomorphic to the variety of full flags in the 3-dimensional space $F_{1,i+2}/F_{1,i-1}$ fixed by $u_{1,i}= u_1|_{ \tilde F_{1,i+2}/\tilde F_{1,i-1}}$. The structure of this variety depends on the type of $u_{1,i}$. There are three possibilities for this type: $(3)$, $(2,1)$ and $(1,1,1)$. Denote $W=\pi'_i(Y_{w,\beta})$, $W'=\pi'_i(Y_{w',\beta'})$. Since $\pi'_i$ is proper, $W$ and $W'$ are closed. Suppose $W=W'$. This means that for a general point $\tilde y\in W$ the fiber $(\pi'_i)^{-1}(\tilde y)$ is reducible, so the type of $u_{1,i}$ equals $(2,1)$. Such fiber has a form of a union of two intersecting projective lines:
\begin{gather*}
(\pi'_i)^{-1}(\tilde y)=l_1(\tilde y)\cup l_2(\tilde y);\\
l_1(\tilde y)=\{(F_1,F_2,v,u_1,u_2,v^*)\in(\pi'_i)^{-1}(\tilde y)\mid F_{1,i}/F_{1,i-1}=\im u_{1,i}\};\\
l_2(\tilde y)=\{(F_1,F_2,v,u_1,u_2,v^*)\in(\pi'_i)^{-1}(\tilde y)\mid F_{1,i+1}/F_{1,i-1}=\ker u_{1,i}\}.
\end{gather*}
Let $U\subset W$ be the set of all $\tilde y\in W$ such that the type of $u_{1,i}$ equals $(2,1)$. It is an open dense subset in $W$. Consider the sets $U_k=\bigcup_{\tilde y\in U}\,l_k(\tilde y)$. The set $U_1\cup U_2= (\pi'_i)^{-1}(U)$ is an open subset of $(\pi'_i)^{-1}(W)$. Since $U_1\cup U_2\subset Y_{w,\b}\cup Y_{w',\b'}$ and $U_k$ are irreducible, we must have either $Y_{w,\b}=\overline{U_1}$, $Y_{w',\b'}=\overline{U_2}$ or
$Y_{w,\b}=\overline{U_2}$, $Y_{w',\b'}=\overline{U_1}$.
Without loss of generality, we can assume that we have the first case. Then it is easy to see that $(w,\b)\in\Phi_{i+1}\setminus\Phi_i$, $(w',\b')\in\Phi_i\setminus\Phi_{i+1}$.

 Conversely, if $(w,\b)\in\Phi_i\setminus\Phi_{i+1}$ (resp.\  $(w,\b)\in\Phi_{i+1}\setminus\Phi_i$) then the type of $u_{1,i}$ for a general $\tilde y\in W$ is $(2,1)$. Therefore, we have $Y_{w,\b}=\overline{U_1}$ (resp.\  $Y_{w,\b}=\overline{U_2}$). So, there exists exactly one $(w',\b')\in\Phi_{i+1}\setminus\Phi_i$ (resp.\  $(w,\b)\in\Phi_i\setminus\Phi_{i+1}$) such that $\pi'_i(Y_{w,\beta})=\pi'_i(Y_{w',\beta'})$.

It is clear that the condition $(w,\b)\in\Phi_{i+1}\setminus\Phi_i$ is equivalent to the fact that the first two parts of one of the conditions of the lemma are satisfied. So, we must prove that if $(w',\b')$ is given by the last two equations of this condition then we have $W=W'$.
Choose a general $y\in Y_{w,\b}$ and let $\tilde y=\pi'_i(y)$. Since $(w,\b)\in\Phi_{i+1}\setminus\Phi_i$, we have $y\in l_1(\tilde y)$. We must prove that $l_2(\tilde y)\subset Y_{w',\b'}$. Let $y=(F_1,F_2,v,u_1,u_2,v^*)$ be given by~\eqref{F1},~\eqref{F2},~\eqref{v} and~\eqref{ukv*} for some basis $\{e_i\}$.

First suppose the condition~1 of the lemma is satisfied. Then we have $a:=(u_1)_{i,i+1} \ne0$, $b:=(u_1)_{i,i+2} \ne0$, $(u_1)_{i+1,i+2} =0$. So,
 $$\ker u_{1,i}=\langle e_i,be_{i+1}-ae_{i+2}\rangle \bmod F_{1,i-1}.$$
Consider the space $F'_{1,i+1}$ such that $F'_{1,i+1}/F_{1,i-1}=\ker u_{1,i}$. Then $$F'_{1,i+1}=g\cdot F_{1,i+1} \text{ where }g=\mathrm{id}-(a/b)E_{i+2,i+1}\in\GL(V).$$
A general point $y_1\in l_2(\tilde y)$ has a form $y_1=(F''_1,F_2,v,u_1,u_2,v^*)$ where $F''_{1,i+1}=F'_{1,i+1}$ and $F''_{1,j}=F_{1,j}$ for $j\ne i,i+1$. If $F''_1\ne g\cdot F_1$ then $F''_1=ghs_i\cdot F_1$ where $h=\mathrm{id}+cE_{i,i+1}$ for some $c\in\sk$, and $s_i\in\GL(V)$ is given by $s_ie_j=e_{s_i(j)}$.

Denote $g'=ghs_i$. Let $y'_1=(g')^{-1}\cdot y_1$. Then
 $$y'_1=\bigl(F_1,\ (g')^{-1}\cdot F_2,\ (g')^{-1}\cdot v,\ (\Ad(g')^{-1})\cdot u_1,\ (\Ad(g')^{-1})\cdot u_2,\ (g')^*v^*\bigr)$$
 Note that $g$ and $h$ preserve $F_2$, so we have
$$(g')^{-1}\cdot F_2= s_ih^{-1}g^{-1} \cdot F_2=s_i\cdot F_2$$
 Therefore $(g')^{-1}F_{2,j}=\langle e_{w'(1)},\dots,e_{w'(j)}\rangle$ for all~$j$. Further,
$$(g')^{-1}v=\sum_{j\in\b} d_js_ie_j=\sum_{j\in\b'} d_{s_i(j)}e_j$$
where $d_j\in\sk$ and $d_j\ne0$ for general $a,b,c$. This implies that $y_1'\in N^*\O_{w',\b'}\subset Y_{w',\b'}$. Hence $y_1=g'\cdot y_1'\in Y_{w',\b'}$.
Thus any general point $y_1\in l_2(\tilde y)$ lies in $Y_{w',\b'}$.
Therefore $l_2(\tilde y)\subset Y_{w',\b'}$. QED.

Other cases can be considered in a similar way.
\end{proof}

\subsection{}
\label{defn K_i}
For each $(w,\beta)\in RB$ there exists at most one  $(w',\beta')\in RB$
such that conditions of the above lemma are satisfied. We will denote
it by $(w',\beta')=K_i(w,\beta)$.

\begin{lem}
\label{3}
Let $W=\pi(Y_{w,\b})$ be the image of an irreducible component of $Y$.
Choose an
open dense subset $U\subset W$ such that the type $\lambda$ of the
nilpotent $u_1$
is the same for all points of $U$. Consider the set $C_W=\{(w',\beta')\in
RB \mid \pi(Y_{w',\beta'})=W\}$. There exists a natural bijection
$\tau_W\colon C_W\to \St(\lambda)$ such that for each $p_0\in U$ and
$(w',\b')\in C_W$ we have $Y_{w',\b '}\cap\pi^{-1}(p_0)=\Fl_{u_1,
\tau_W(w',\b')}\times\{p_0\}$.
\end{lem}

\begin{proof} For each $T\in\St(\l)$ consider the set
\eql{UT}{
U_T=\bigcup_{p\in U}\,\bigl(\Fl_{u_1(p),T}\times\{p\}\bigr)
\subset\pi^{-1}(U)\subset Y.}
We have $\bigcup_{T\in\St(\l)}\,U_T=\pi^{-1}(U)$. The sets $U_T$ are
irreducible components of $\pi^{-1}(U)$. Note that the equation
$\pi(Y_{w',\b'})=W$ is equivalent to the fact that $Y_{w',\b'}$
dominates $W$ (we use that $\pi$ is proper). In this case $Y_{w',\b'}$
must coincide with $\overline{U_T}$ for some~$T$. In particular,
$Y_{w,\b}=\overline{U_{T_0}}$ for some $T_0\in\St(\l)$.
Since $\dim U_T=\dim U_{T_0}=\dim Y_{w,\b}=\dim Y$ for each
$T\in\St(\l)$, each $\overline{U_T}$ is an irreducible component of $Y$
such that $\pi(\overline{U_T})=W$. So, we have a one-to-one correspondence
between the sets $C_W$ and $\St(\l)$. Obviously, this correspondence can
be described as in the statement of the lemma.
\end{proof}

\subsection{}
Let $\pi_i\colon \Fl(V)\to\Fl^{(i)}(V)$
be the natural projection.
For each $i\in\{1,\dots,N-1\}$ consider the set
$\Phi'_i\subset \St(\lambda) $
defined as follows:
$T\in \Phi'_i$ iff for a general
$F\in \Fl_{u_1,T}$ we have $u_1|_{F_{i+1}/F_{i-1}}=0$.

\begin{lem}
\label{4}
a) Let $T\in \St(\lambda)$. Then $T\in\Phi'_i\iff r_i(T)<r_{i+1}(T)\iff
c_i(T)\ge c_{i+1}(T)$ where $r_i(T)$ \textup(resp.\ $c_i(T)$\textup) stands for the number of row (resp.\ column) in $T$ containing $i$.

b) Let $T,T'\in \St(\lambda)$ and $T\ne T'$.
Then $\pi_i(\Fl_{u_1,T})=\pi_i(\Fl_{u_1,T'})$
iff one of the following conditions is satisfied up to interchanging
$T$ and $T':$

\noindent
$1$.\footnote{the condition $r_{i+2}(T)\le r_i(T)<r_{i+1}(T)$ is
equivalent to  $c_{i+2}(T)> c_i(T)\ge c_{i+1}(T)$.}
$r_{i+2}(T)\le r_i(T)<r_{i+1}(T)$ and $T'$ is obtained from $T$
by interchanging $i+1$ and $i+2$.

\noindent
$2$.\footnote{the condition $r_i(T)< r_{i+2}(T)\le r_{i+1}(T)$
is equivalent to  $c_i(T)\ge c_{i+2}(T)> c_{i+1}(T)$.}
$r_i(T)< r_{i+2}(T)\le r_{i+1}(T)$ and $T'$ is obtained
from $T$ by interchanging $i$ and $i+1$.
\end{lem}

\begin{proof}
a) This statement is equivalent to Lemma~5.11 in~\cite{S}.

b) Arguments similar to those used in the proof of Lemma~\ref{2}~b) show that we can define an involution $K'_i\colon \Phi'_i\btu\Phi'_{i+1}\to \Phi'_i\btu\Phi'_{i+1}$ such that $K'_i(\Phi'_i\setminus\Phi'_{i+1})=\Phi'_{i+1}\setminus\Phi'_i$ and such that a pair of tableaux $T\ne T'\in\St(\l)$ satisfies $\pi_i(\Fl_{u_1,T})=\pi_i(\Fl_{u_1,T'})$ iff $T\in\Phi'_i\btu\Phi'_{i+1}$ and $T'=K'_i(T)$. Thus we must prove that this involution can be described by the conditions~1,~2 of the lemma.

Suppose $T\in\Phi'_i\setminus\Phi'_{i+1}$ and $T'=K'_i(T)\in\Phi'_{i+1}\setminus\Phi'_i$. Then the first part of one of the conditions~1,~2 must be satisfied, and we must prove that $T'$ is given by the second part. The equation $\pi_i(\Fl_{u_1,T})=\pi_i(\Fl_{u_1,T'})$ implies that for $j\in\{0,\dots,N\}\setminus\{i,i+1\}$ and for general $F\in\Fl_{u_1,T}$ and $F'\in\Fl_{u_1,T'}$ the types of $u_1|_{F_j}$ and $u_1|_{F'_j}$ are the same. This means that $T$ and $T'$ can differ only in the position of $i,\ {i+1},\ {i+2}$.

Moreover, choose a general point $\tilde F\in\pi_i(\Fl_{u_1,T})$. Let $l_k(\tilde F)\ (k=1,2)$ be two irreducible components of $\pi_i^{-1}(\tilde F)$ defined similarly to the proof of Lemma~\ref{2}~b), and let $F',F''$ be general points of $l_1(\tilde F),l_2(\tilde F)$ respectively. Then $F'$ (resp.\ $F''$) is a general point of $\Fl_{u_1,T'}$ (resp.\ $\Fl_{u_1,T}$). In particular, $F'\in\Fl_{u_1}^{T'},\ F''\in\Fl_{u_1}^T$. Let $l_1(\tilde F)\cap l_2(\tilde F)=\{F_1\}$, and let $T_1$ be the tableau such that $F_1\in\Fl_{u_1}^{T_1}$. Then we obtain that $T$ (resp.\ $T'$) can differ from $T_1$ only in the position of $i$ and $i+1$ (resp.\ $i+1$ and $i+2$).

If $T$ satisfies the first part of the condition~2 of the lemma, the last condition and the conditions $T\in\Phi'_i\setminus\Phi'_{i+1}$ and $T'\in\Phi'_{i+1}\setminus\Phi'_i$ imply the desired statement.

If the first part of the condition~1 is satisfied, there is another a~priori possible case:
$$
\left\{\begin{aligned}
&r_i(T_1)<r_{i+1}(T_1)<r_{i+2}(T_1)\\
&c_i(T_1)>c_{i+1}(T_1)>c_{i+2}(T_1)\\
&\text{$T$  is obtained from~$T_1$ by interchanging $i$ and $i+1$ }\\
&\text{$T'$ is obtained from~$T_1$ by interchanging $i+1$ and $i+2$}
\end{aligned}
\right.
$$
In this case consider the tableau $T''$ obtained from~$T$ by interchanging $i+1$ and $i+2$. If we apply the above argument to the pair $K'_i(T''),T''$ instead of $T,T'$, we will obtain that $K'_i(T'')=T$, contradicting $K'_i(T')=T$. So, the lemma is proved.
\end{proof}

\subsection{}
For each $T\in \St(\lambda)$ there exists at most one tableau $T'$
satisfying the conditions of Lemma~\ref{4}. Denote it by $T'=K'_i(T)$.

\begin{lem}
\label{5}
Let $W$ be the image of an irreducible component of $Y$ under the map $\pi$.
Then for each $i\in\{1,\dots,N-1\}$ we have $\tau_W(\Phi_i\cap
C_W)\subset\Phi'_i$ and
 for each $i\in\{1,\dots,N-2\}$ we have $\tau_W\circ K_i=K'_i\circ\tau_W$.
 \end{lem}

\begin{proof}
In the proof of Lemma~\ref{3} we have shown that for each $(w,\b)\in C_W$
we have $Y_{w,\b}=\overline{U_T}$ where $T=\tau_W(w,\b)\in\St(\l)$.
Now the first inclusion follows immediately from the definition of $\Phi_i$
and $\Phi'_i$. Let us prove the second equation. Suppose $(w',\b')=K_i(w,\b)$.
Let $\tilde\pi_i\colon Y_i\to\tilde Y$ be the projection satisfying
$\tilde\pi_i\circ\pi'_i=\pi$. Then $$W=\pi(Y_{w,\b})=
\tilde\pi_i(\pi'_i(Y_{w,\b}))=\tilde\pi_i(\pi'_i(Y_{w',\b'}))=\pi(Y_{w',\b'})$$
 So, $(w',\b')\in C_W$. Denote $T=\tau_W(w,\b),\ T'=\tau_W(w',\b')$.
Then for each $p_0\in U$ we have
$$\pi_i(\Fl_{u_1(p_0),T})= \pi_i(Y_{w,\b}\cap\pi^{-1}(p_0))=
\pi'_i(Y_{w,\b})\cap\tilde\pi_i^{-1}(p_0)$$
 and the same for~$T'$. The equation $(w',\b')=K_i(w,\b)$ means that
$\pi'_i(Y_{w,\b})= \pi'_i(Y_{w',\b'})$ and $(w,\b)\ne(w',\b')$.
Taking the intersection with $\pi^{-1}(p_0)$, we get $\pi_i(\Fl_{u_1,T})=
\pi_i(\Fl_{u_1,T'})$ and $T\ne T'$ (this inequality follows from the fact
that $\tau_W$ is bijective).

 So, we get $T'=K'_i(T)$.
\end{proof}

\subsection{} We write down the action of Hecke algebra generators on
bimodule $\CR$ in the Kazhdan-Lusztig basis $\utH_\tw$.
For $i\in\{1,\ldots,N-1\}$, recall the subset $\Phi_i\subset RB$ introduced
in~\ref{defn Phi_i}.

Let $\tw,\tw'\in RB$ and $\tw'<\tw$. Consider the restriction $IC(\zowt{})|_{\O_{\tw'}}$
of the $IC$-sheaf of~$\zowt{}$ to~$\O_{\tw'}$. It is a constant Tate complex on~$\O_{\tw'}$
concentrated in cohomological degrees less than $-\bn-\ell(\tw')$. We denote by $\mu(\tw',\tw)=
\mu(\tw,\tw')$ the dimention $\dim H^{-\bn-\ell(\tw')-1}(IC(\zowt{})_x)$ where $x\in\O_{\tw'}$.

\begin{prop}\label{HKL}
For any $\tw\in RB$ and $i\in{1,\ldots,N-1}$ we have
$$
\utH_\tw\cdot\utH_{s_i}=
\begin{cases}
-(\bv^{-1}+\bv)\utH_\tw, &
\textrm{if}\quad \tw\in\Phi_i  \\
\utH_{\tw*s_i} + \sum\limits_{\substack{\tw'<\tw\\\tw'\in\Phi_i}}
\mu(\tw',\tw) \utH_\tw'
& \textrm{if}\quad \tw\notin\Phi_i\\
\end{cases}
$$
\end{prop}

\proof
 By definition, $\utH_\tw$ is the class of the $IC$-sheaf of the orbit
closure $\zowt{}$,\quad $\utH_\tw=[IC(\zowt{})]$.
Therefore $\utH_\tw\cdot\utH_{s_i}$ is the class of the
direct image  of the $IC$-sheaf
under the map $\psi\colon S=\zos{i}\times_{{}_{\Fl(V)}}\zowt{}\to X$.
If $\tw\in\Phi_i$ then the image of this map coincides with $\zowt{}$ and all
its fibers are isomorphic to $\PP^1$, hence we obtain the required formula.
If $\tw\notin\Phi_i$ then $\im(\psi)=\zowts\quad(s=s_i)$ and all the fibers
of~$\psi$ are isomorphic either to $\PP^1$ or to a point. We  claim that
the direct image  of the $IC$-sheaf under the map $\psi$ is perverse. Indeed,
pick an orbit~$\O_{\tw'}$ inside $\zowts$. We need  to show that $\psi_*(
IC(S))|_{\O_{\tw'}}$ is concentrated in degrees $\le -\bn-\ell(\tw')$.
Let $x\in \O_{\tw'}$ and $Q=\psi^{-1}(x)$. Then we have $\psi_*(IC(S))_x=
R\Gamma(Q,IC(S)|_Q)$. If $\tw'=\tw*s$ then $Q$ is a point and $R\Gamma(Q,
IC(S)|_Q)=\ol\QQ_l[\bn+\ell(\tw')]$. If $Q$ is a point but $\tw'\ne\tw*s$
then the properties of $IC$-sheaf imply that $H^m(IC(S)|_Q)=0$
for $m\ge -\bn-\ell(\tw')$.

Otherwise, if $Q\cong\PP^1$ (this happens if and only if $\tw'*s<\tw$),
let $U= Q\cap\phi^{-1}(\O_{\tw'*s})$ where $\phi\colon S \to\zowt{}$
is the projection
to the second factor. Then $U$ is open and dense in~$Q$,
and $IC(S)|_U$ is constant.
The complex $IC(S)|_U$ is concentrated
in degrees $\le-\bn-\ell(\tw')-2$, and $IC(S)|_{Q\setminus U}$ is concentrated
in degrees $\le-\bn-\ell(\tw')-1$. From this we obtain that
$H^m(IC(S)|_Q) =0$ for
$m>-\bn-\ell(\tw')$ and $\dim H^{-\bn-\ell(\tw')}(IC(S)|_Q)=
\dim H^{-\bn-\ell(\tw')-2}
(IC(S)_{x'})$ where $x'\in U$. Note that if we identify $U$ with
$\phi(U)$, we have
$IC(S)|_U= IC(\zowt{})|_{\phi(U)}[1]$. Besides $\phi(U)\subset\O_{\tw'*s}$.
If $\tw'\notin\Phi_i$ then $\tw'*s>\tw'$, and therefore
$\dim H^{-\bn-\ell(\tw')-2}
(IC(S)_{x'})=\dim H^{-\bn-\ell(\tw'*s)}(IC(\zowt{})_{\phi(x')})=0$.
If $\tw'\in\Phi_i$ then $\tw'*s=\tw'$, and therefore
$\dim H^{-\bn-\ell(\tw')-2}
(IC(S)_{x'})=\dim H^{-\bn-\ell(\tw*s)-1}(IC(\zowt{})_{\phi(x')})=
\mu(\tw'*s,\tw)=
\mu(\tw',\tw)$.
So we get
\begin{gather*}
 \dim H^m(\psi_*(IC(S))_x)= \dim H^m(IC(S)|_Q)=0 \quad
   \text{if $m>-\bn-\ell(\tw')$;} \\
  \begin{split}
 \dim H^{-\bn-\ell(\tw')}(\psi_*(IC(S))_x) &= \dim H^{-\bn-\ell(\tw')}
(IC(S)|_Q)\\
	&=
  \begin{cases}
    1 &\text{if $\tw'=\tw*s$;} \\
    \mu(\tw',\tw) &\text{if $\tw'<\tw$ and $\tw\in\Phi_i$;} \\
    0 &\text{otherwise.}
  \end{cases}
   \end{split}
\end{gather*}
Now, taking in account that $\psi_*(IC(S))$ is selfdual,
we obtain the desired decomposition.
\qed

\begin{Rem} \label{W-graph}
Note that Proposition~\ref{HKL} implies that the bimodule~$\R$ arises from a
certain $\gS_N \times\gS_N^0$-graph~$\Gamma_{\mir}$ in the
terminology of~\cite{KL},
where $\gS_N^0 \cong\gS_N$ is the opposed group to~$\gS_N$, i.~e.\
$\gS_N^0=\{ g^0,\ g\in\gS_N \}$ with multiplication given by $g^0 h^0= (hg)^0$.
The set of vertices of~$\Gamma_{\mir}$ is~$RB$;
the labels~$I_\tw$ are defined by
$I_\tw = \{\,s_i^0 \mid \tw\in\Phi_i\,\} \,\cup\,
 \{\,s_i \mid \tw^{-1}\in\Phi_i\,\}$;
the edges are $\{\tw,\tw'\}$ such that $\tw'<\tw$, $\mu(\tw',\tw)\ne0$ and
$I_\tw \ne I_{\tw'}$;
finally, the multiplicities are $\mu(\tw,\tw')$.
\end{Rem}

\subsection{One-sided microlocal cells}
Let $W=\pi(Y_{\tw})$ \quad $(\tw\in RB)$ be the image of an irreducible
component of $Y$.
We define the {\itshape right microlocal cell} corresponding to $W$ as the
set $C_W$ described in Lemma~\ref{3}.
We define a {\em left microlocal cell} as the image of a right microlocal
cell under the involution $\tw\mapsto\tw^{-1}$.
In terms of bijection $\RSK$ introduced in~\ref{RSKmir}, two-sided microlocal
cells are given by condition
$\bt(\tw)=const$. The left microlocal cells  are given by conditions
$\bt(\tw)=const$
and $T_1(\tw)=const$ , and the right microlocal cells are given by conditions
$\bt(\tw)=const$ and $T_2(\tw)=const$. Each two-sided microlocal
cell is a union of left microlocal cells and of right microlocal cells as well;
moreover, each left microlocal cell and right microlocal cell inside
the same two-sided microlocal cell
intersect exactly in one element. Two-sided microlocal cells are the
minimal subsets which are unions of both left and right microlocal cells.

Now recall that $\bt(\tw)=(\nu,\theta,\nu')=:(\nu(\tw),\theta(\tw),\nu'(\tw))$.

\begin{thm}
\label{L}
a) Each left (right, two-sided) microlocal cell is contained in
a left (resp. right, bimodule) Kazhdan-Lusztig cell.

b) Conversely, for $\tw_1,\tw_2$ in the same left (right, bimodule)
Kazhdan-Lusztig
cell, we have $\nu(\tw_1)=\nu(\tw_2)$ (resp. $\nu'(\tw_1)=\nu'(\tw_2)$,
resp. $\nu(\tw_1)=\nu(\tw_2)$ and $\nu'(\tw_1)=\nu'(\tw_2)$).
\end{thm}

\begin{proof}
It suffices to prove the theorem for one-sided cells and, by the reason
of symmetry, only for right-handed ones.
Let us formulate the following auxiliary proposition.

\begin{prop}\label{2el}
Two elements $\tw,\tw'$ lie in the same right microlocal cell iff there is
a sequence $\tw=\tw_1,\tw_2,\ldots,\tw_m=\tw'$ such that for each
$j=1,\ldots,m-1$ there is $i\in\{1,\ldots,N-2\}$ such that
$\tw_{j+1}=K_i(\tw_j)$ (see~\ref{defn K_i}).
\end {prop}

\begin{proof}
It is easy to see from the definition of operations $K_i$ that if
$\tw'=K_i(\tw)$ then $\tw$ and $\tw'$ lie in the same right microlocal cell.
This implies the ``only if'' direction. Conversely, let $\tw$ and
$\tw'$ lie in
a microlocal cell corresponding to $W$. Consider the bijection
$\tau_{W}: C_W\to\St(\l)$ of Lemma~\ref{3}. In view of Lemma~\ref{5}
it suffices to prove that any 2 standard Young tableaux of the same shape
can be obtained from each other by a successive application of
operations $K'_i$.

It can be checked directly.
\end{proof}

It is easy to check that if $\tw'=K_i(\tw)$ then, up to permutation of
$\tw$ and $\tw'$, we have $\tw\lessdot\tw'$, moreover
$$
\begin{cases}
\tw'=\tw*s_i  \\
\tw\in\Phi_{i+1} \\
\tw'\notin\Phi_{i+1} \\
 \end{cases}
\textrm{or}
 \quad\quad
\begin{cases}
\tw'=\tw*s_{i+1}  \\
\tw\in\Phi_i \\
\tw'\notin\Phi_i \\
 \end{cases}
 $$

Observe that if $\tw<\tw'$ and $\ell(\tw')=\ell(\tw)+1$
then $\mu(\tw,\tw')=1$, so,
taking in account Proposition~\ref{HKL}, it follows that if $\tw'=K_i(\tw)$
then $\tw'$ and $\tw$ lie in the same right Kazhdan-Lusztig cell.
Therefore, in view of Proposition~\ref{2el}, each microlocal cell lies in a
Kazhdan-Lusztig cell.
So the proof of Theorem~\ref{L}.a) is finished.

For the proof of b), we can realize the $\mathcal H$-bimodule $\mathcal R$
in the Grothendieck group of $G$-equivariant Hodge $D$-modules on $X$.
Then we have the functor of singular support from the category of
$G$-equivariant Hodge $D$-modules on $X$ to the category of $G$-equivariant
coherent sheaves on $T^*X$ supported on $Y$.
Similarly, we have the functor of singular support
from the category of $G$-equivariant Hodge $D$-modules on $\Fl(V)\times\Fl(V)$
to the category of $G$-equivariant coherent sheaves on the Steinberg variety
of $G$. These functors are compatible with the convolution action.
Thus if $IC_{\tw_1}$ is a direct summand of
the convolution of $IC_{\tw_2}$ with
$IC_w$, and $(\ ^1u_1,\ ^1u_2,\ ^1v,\ ^1v^*)$ (resp.
$(\ ^2u_1,\ ^2u_2,\ ^2v,\ ^2v^*)$) is a general element in the conormal bundle
to $\Omega_{\tw_1}$ (resp. $\Omega_{\tw_2}$), then $^1u_1$ must lie in the
closure of $G$-orbit of $^2u_1$ (and similarly, $^1u_2$ must lie in the closure
of $G$-orbit of $^2u_2$). The proof of b) is completed.
\end{proof}

\subsection{Fourier duality} \label{FD}
In this subsection we will write $X(V),Y(V),\O_\tw(V),\dots$ instead of
$X,Y,\O_\tw,\dots$ to emphasize the dependence on~$V$.
All the statements in this subsection are straightforward and
left to the reader as an exercise.

Note that we have a canonical isomorphism $Y(V)\cong Y(V^*)$,
$(F_1,F_2,v,u_1,u_2,v^*)\mapsto(F_1^*,F_2^*,v^*,u_1^*,u_2^*,v)$. Therefore
we get a bijection between the sets of their irreducible components, which
gives rise to an involution~$\F$ on $RB$.

\begin{prop}\label{F:RB}
For any $\tw=(w,\b)\in RB$ we have $\F(\tw)=(w_0ww_0,\{1,\dots,N\}\setminus
w_0(\b))$ where $w_0\in\gS_N$ is the longest element, i.~e.\ $w_0(i)=N+1-i$.
\qed\end{prop}

Further, we have an isomorphism $\psi\colon Z(V)\stackrel\sim\to Z(V^*)$.
It carries images of irreducible
components of~$Y(V)$ to images of irreducible components of~$Y(V^*)$,
therefore $\psi(Z_\bt(V))=Z_{\bt^*}(V^*)$ for some $\bt^*\in\bT$.

\begin{prop}\label{F:T}
If $\bt=(\nu,\theta,\nu')\in\bT$ then $\bt^*=(\nu,\theta^*,\nu')$ where
$\theta^*_i=\min\{\nu_i,\nu_i'\}+
\max\{\nu_{i+1},\nu_{i+1}'\}-\theta_i$.
\qed\end{prop}

\begin{prop}\label{F:RSK}
If $\RSK(\tw)=(\bt,T_1,T_2)$ then $\RSK(\F(\tw))=(\bt^*,T_1^*,T_2^*)$
where $\bt^*$ is the same as in Proposition%
~\ref{F:T},  and $T_1^*$, $T_2^*$ are conjugate tableaux to $T_1$, $T_2$
(see~\cite{Fulton} for the definition).
Besides, the partition $\theta^*(\tw)=\theta(\F(\tw))$ is the shape of the
tableau~$T^@_N$ from~\ref{RSK2}
with all @'s removed.
\qed\end{prop}

\begin{cor} \label{F:micro.cells}
The involution~$\F$ on~$RB$ carries left, right, and two-sided microlocal
cells to left, right, and two-sided microlocal cells,
respectively.
\qed\end{cor}

Now consider the Fourier-Deligne transform~$\FD$ from the derived
constructible $G$-equivariant category of $\ol{\mathbb Q}_l$-sheaves on
$X(V)=\Fl(V)\times\Fl(V)\times V$ to the derived constructible $G$-equivariant
category of $\ol{\mathbb Q}_l$-sheaves on
$X(V^*)=\Fl(V^*)\times\Fl(V^*)\times V^*
\cong\Fl(V)\times\Fl(V)\times V^*$. (We use the canonical identification
$\Fl(V)\cong\Fl(V^*)$ (by taking annihilators) in the first two factors, and the
Fourier-Deligne transform in the third factor.) It gives rise to
an involution~$\F$ on~$\R$
which is compatible with the automorphism of the algebra~$\H$
induced by conjugation with~$w_0$ on the Coxeter group~$\gS_N$. It carries
$G$-equivariant $IC$-sheaves on~$X(V)$ to $G$-equivariant
$IC$-sheaves on~$X(V^*)$. Therefore we obtain the following

\begin{prop}\label{F:KLbasis}
For any $\tw\in RB$ we have $\F(\utH_\tw)=\utH_{\F(\tw)}$.
\qed\end{prop}

\begin{cor} \label{F:KLcells}
The involution~$\F$ on~$RB$ carries left, right, and bimodule
Kazhdan-Lusztig cells to left, right, and bimodule Kazhdan-Lusztig cells,
respectively.
\qed\end{cor}

\subsection{Relation to mirabolic character sheaves}
\label{char}
Recall the definition of unipotent mirabolic character sheaves on
$\GL(V)\times V$, cf.~\cite{FG2}~4.1 and~\cite{FGT}~5.2.
We consider the following diagram of $\GL(V)$-varieties and
$\GL(V)$-equivariant maps:
$$\GL(V)\times V\stackrel{pr}{\longleftarrow}\GL(V)\times\Fl(V)\times V
\stackrel{f}{\longrightarrow}\Fl(V)\times\Fl(V)\times V.$$
In this diagram, the map $pr$ is given by $pr(g,x,v):=(g,v)$, while
the map $f$ is given by $f(g,x,v):=(gx,x,gv)$. The group $\GL(V)$ acts
diagonally on all the product spaces in this diagram, and acts on itself
by conjugation.

The functor $\mathsf{CH}$ from the constructible derived category of $l$-adic
sheaves on $\Fl(V)\times\Fl(V)\times V$ to the constructible derived category
of $l$-adic sheaves on $\GL(V)\times V$ is defined as
$\mathsf{CH}:=pr_*f^![-\dim\Fl(V)]$.
Now let $\mathcal F$ be a $\GL(V)$-equivariant perverse sheaf on
$\Fl(V)\times\Fl(V)\times V$. The irreducible perverse constituents of
$\mathsf{CH}\mathcal F$ are called unipotent mirabolic character sheaves
on $\GL(V)\times V$. Clearly, this definition is a direct analogue of
Lusztig's definition of character sheaves.

Recall the following examples of unipotent mirabolic character sheaves
(see~\cite{FG}~5.4). For $M\leq N$ let $\tilde{\mathfrak X}_{N,M}$ be
a smooth variety of triples $(g,F_\bullet,v)$ where $g\in\GL(V)$, and
$F_\bullet\in\Fl(V)$ is a complete flag preserved by $g$, and $v\in F_M$.
We have a proper morphism $\pi_{N,M}:\ \tilde{\mathfrak X}_{N,M}\to\GL(V)\times
V$ (forgetting $F_\bullet$) with the image ${\mathfrak X}_{N,M}\subset
\GL(V)\times V$ formed by all the pairs $(g,v)$ such that
$\dim\langle v,gv,g^2v,\ldots\rangle\leq N-M$. According to Corollary~5.4.2
of {\em loc. cit.}, we have
$$(\pi_{N,M})_*IC(\tilde{\mathfrak X}_{N,M})\simeq
\bigoplus_{|\mu|=M}^{|\lambda|=N-M}
L_\mu\otimes L_\lambda\otimes{\mathcal F}_{\lambda,\mu}$$
for certain unipotent mirabolic character sheaves ${\mathcal F}_{\lambda,\mu}$
supported on ${\mathfrak X}_{N,M}$, and $L_\lambda$, resp. $L_\mu$, is
an irreducible representation of ${\mathfrak S}_{N-M}$,
resp. ${\mathfrak S}_M$.

We conjecture the following formula for the class of
$\mathsf{CH}\utH_\tw$ in the $K$-group of unipotent mirabolic Weil sheaves.

\begin{conj}
\label{0101}
$\mathsf{CH}\utH_\tw=\sum_{|\lambda|+|\mu|=N}
f_{\lambda,\mu}(\utH_\tw)[{\mathcal F}_{\lambda,\mu}]$
where $f_{\lambda,\mu}$ is a functional $\CR\to\BZ[\bv,\bv^{-1}]$ such
that $f_{\lambda,\mu}(hr)=f_{\lambda,\mu}(rh)$ for any $r\in\CR,\ h\in\CH$.
Moreover, in the decomposition~(\ref{f3}) of Proposition~\ref{algH},
$f_{\lambda,\mu}$ vanishes on all the summands except for
$V^*_\nu\otimes V_\nu$ corresponding to
$(\tilde{\nu},\tilde{\theta},\tilde{\nu})\in{\mathbf T}$
where $\Upsilon(\lambda,\mu)=(\nu,\theta)$.
\end{conj}

\subsection{Asymptotic bimodule}
\label{asymp}
For a partition $\nu$ of $N$, let $c_\nu\subset{\mathfrak S}_N$ be
the corresponding two-sided KL cell. Let
$a(c_\nu)=a_\nu=N^2-N-n_\nu:=\frac{N^2-N}{2}-\sum_{i\geq1}(i-1)\nu_i$
be its $a$-function.
For multiplication in $\CH$ we have $\utH_w\cdot\utH_y=
\sum_{z\in{\mathfrak S}_N}m_{w,y,z}\utH_z$,
for $m_{w,y,z}\in\BZ[\bv,\bv^{-1}]$. If $w,y,z\in c_\nu$ then, according to
Lusztig, the degree of $m_{w,y,z}$ is less than or equal to $a_\nu$.
Let $\gamma_{w,y,z}\in\BZ$ be the coefficient of $\bv^{a_\nu}$ in
$m_{w,y,z}$. Lusztig's asymptotic ring $J_\nu$ is defined as a ring
with a basis $\{t_w,\ w\in c_\nu\}$ and multiplication
$t_w\cdot t_y=\sum_{z\in c_\nu}\gamma_{w,y,z}t_z$.
By the classical RSK algorithm, $c_\nu$ is in bijection with the set of
pairs of standard tableaux $\{(T_1,T_2)\}$ of shape $\nu$.
According to~\cite{Lus}~3.16.b), the ring $J_\nu$ with basis $\{t_w\}$ is
isomorphic to the matrix ring $\operatorname{Mat}_{\operatorname{St}(\nu)}$
with the basis of elementary matrices $\{e_{T_1,T_2}\}$, so that
$t_w$ goes to $e_{T_1,T_2}$ where $(T_1,T_2)$ are constructed from $w$ by
the classical RSK algorithm.

Now for a pair of partitions $\nu\supset\theta$ we consider the corresponding
bimodule KL cell $c_{\nu\supset\theta\subset\nu}\subset RB$.
For $\tw\in RB$, and $y\in{\mathfrak S}_N$ we have
$\utH_\tw\cdot\utH_y=\sum_{\tilde{z}\in RB}m_{\tw,y,\tilde{z}}\utH_{\tilde z}$,
and
$\utH_y\cdot\utH_\tw=\sum_{\tilde{z}\in RB}m_{y,\tw,\tilde{z}}\utH_{\tilde z}$.
We conjecture that for $\tw,\tilde{z}\in c_{\nu\supset\theta\subset\nu}$, and
$y\in c_\nu$, the degrees of $m_{\tw,y,\tilde{z}}$ and
$m_{y,\tw,\tilde{z}}$ are less than or equal to $a_\nu$.
We denote by $\gamma_{y,\tw,\tilde{z}}$ the coefficient of
$\bv^{a_\nu}$ in $m_{y,\tw,\tilde{z}}$, and we
denote by $\gamma_{\tw,y,\tilde{z}}$ the coefficient of
$\bv^{a_\nu}$ in $m_{\tw,y,\tilde{z}}$.
We define the asymptotic bimodule $J_{\nu\supset\theta\subset\nu}$
over $J_\nu$ as a bimodule
with a basis $\{t_\tw,\ \tw\in c_{\nu\supset\theta\subset\nu}\}$
and the action
$t_\tw\cdot t_y=\sum_{\tilde{z}\in
c_{\nu\supset\theta\subset\nu}}\gamma_{\tw,y,\tilde{z}}t_{\tilde z}$, and
$t_y\cdot t_\tw=\sum_{\tilde{z}\in
c_{\nu\supset\theta\subset\nu}}\gamma_{y,\tw,\tilde{z}}t_{\tilde z}$.

\begin{conj}
\label{0201}
The based bimodule $J_{\nu\supset\theta\subset\nu},
\{t_\tw,\ \tw\in c_{\nu\supset\theta\subset\nu}\}$ is isomorphic to the based
regular bimodule $\operatorname{Mat}_{\operatorname{St}(\nu)},
\{e_{T_1,T_2}\}$, so that $t_\tw$ goes to $e_{T_1,T_2}$ where $(T_1,T_2)$
are constructed from $\tw$ by the mirabolic RSK algorithm.
\end{conj}



\begin{thebibliography}{99}

\bibitem{AH} P.~N.~Achar, A.~Henderson, {\em Orbit closures in the
enhanced nilpotent cone}, Adv. in Math. {\bf 219} (2008), 27--62.

\bibitem{B} J.~Bernstein, {\em $P$-invariant distributions on ${\GL}(N)$
and the classification of unitary representations of ${\GL}(N)$
(non-Archimedean case)},  Lie group representations, II (College
Park, Md., 1982/1983), Lecture Notes in Math. {\bf 1041} (1984),
50--102.

\bibitem{Be} J.~Bernstein, {\em Lectures on $D$-modules} (1984),
available at
http://www.math.uchicago.edu/$\widetilde{\hphantom{m}}$mitya/langlands.html

\bibitem{BGS} A.~Beilinson, V.~Ginzburg, W.~Soergel,
{\em Koszul duality patterns in the representation theory},
Journal of the Amer. Math. Soc. {\bf 9} (1996), 473--527.

\bibitem{EG} P.~Etingof, V.~Ginzburg, {\em
Symplectic reflection algebras, Calogero-Moser space, and deformed
Harish-Chandra homomorphism},  Invent. Math.  {\bf 147}
(2002),  no. 2, 243--348.

\bibitem{FG} M.~Finkelberg, V.~Ginzburg, {\em
Cherednik algebras for algebraic curves}, preprint math.RT/0704.3494.

\bibitem{FG2} M.~Finkelberg, V.~Ginzburg, {\em On mirabolic D-modules},
preprint math.AG/0803.0578.

\bibitem{FGT} M.~Finkelberg, V.~Ginzburg, R.~Travkin,
{\em Mirabolic affine Grassmannian and character sheaves}, preprint
math/0802.1652.

\bibitem{Fulton} W.~Fulton, {\em Young tableaux}, London Math. Soc. Student
Texts {\bf 35}, Cambridge Univ. Press (1997), x+260pp.

\bibitem{G} D.~Gaitsgory, {\em Geometric Langlands correspondence
for ${\GL}_n$}, Proceedings of the International Congress of
Mathematicians, Vol. II (Beijing, 2002),
Higher Ed. Press, Beijing (2002), 571--582.

\bibitem{KL} D.~Kazhdan, G.~Lusztig, {\em Representations of Coxeter
groups and Hecke algebras}, Inventiones Math. {\bf 53} (1979), 165--184.


\bibitem{Lus} G.~Lusztig, {\em Leading coefficients of character values
of Hecke algebras}, Proc. Symp. in Pure Math. {\bf 47} (1987), 235--262.

\bibitem{Lu} G.~Lusztig, {\em Introduction to Quantum Groups},
Progr. Math. {\bf 110}, Birkh{\"a}user Boston, Boston, MA (1993).

\bibitem{L} G.~Lusztig, {\em Character sheaves and generalizations},
The unity of mathematics,
Progr. Math. {\bf 244}, Birkh{\"a}user Boston, Boston, MA (2006), 443--455.

\bibitem{M} P.~Magyar, {\em Bruhat Order for Two Flags and a Line},
J. of Alg. Combinatorics {\bf 21} (2005), 71--101.

\bibitem{MWZ} P.~Magyar, J.~Weyman, A.~Zelevinsky, {\em Multiple flags
of finite type}, Adv. in Math. {\bf 141} (1999), 97--118.

\bibitem{So} W.~Soergel, {\em Kazhdan-Lusztig-Polynome und eine
Kombinatorik f{\"u}r Kipp-Moduln}, Representation Theory, Electronic
Journal of Amer. Math. Soc., {\bf 1} (1997), 37--68.

\bibitem{Sol} L.~Solomon, {\em The affine group I, Bruhat decomposition},
J.~Algebra, {\bf 20} (1972), 512--539.

\bibitem{S} N.~Spaltenstein, {\em Classes unipotentes et sous-groupes de
Borel}, Lecture Notes in Math. {\bf 946} (1982), ix+259pp.

\bibitem{Sp} T.~Springer, {\em A purity result for fixed point varieties
in flag manifolds}, J. Fac. Sci. Univ. Tokyo Sect. IA Math. {\bf 31}
(1984), 271--282.

\bibitem{Su} H.~Sumihiro, {\em Equivariant completion}, J. Math. Kyoto
Univ. {\bf 14} (1974), 1--28.

\bibitem{Z} A.~Zelevinsky, {\em Representations of finite classical groups.
A Hopf algebra approach}, Lecture Notes in Math. {\bf 869} (1981), iv+184pp.

\end{thebibliography}
\end{document}